\documentclass[a4paper,11pt]{article}
\usepackage[latin1]{inputenc}
\usepackage[T1]{fontenc}
\usepackage[british,UKenglish,USenglish,american]{babel}
\usepackage{amsmath}
\usepackage{amsfonts}
\usepackage{hyperref}
\usepackage[T1]{fontenc} 
\usepackage[latin1]{inputenc}
\usepackage{theorem}
\usepackage{graphics}
\usepackage{makeidx}
\usepackage{fancyhdr}
\usepackage{hyperref}
\usepackage{verbatim}
\usepackage{bbm}
\usepackage{bm}
\usepackage{amssymb}
\usepackage{color}
\numberwithin{equation}{section}
\usepackage{amssymb}
\newcommand{\be}{\begin{equation}}
\newcommand{\ee}{\end{equation}}
\newcommand{\benn}{\begin{equation*}}
\newcommand{\eenn}{\end{equation*}}
\newcommand{\bea}{\begin{eqnarray}}
\newcommand{\eea}{\end{eqnarray}}
\newcommand{\diff}{\txtd}
\newcommand{\beann}{\begin{eqnarray*}}
\newcommand{\eeann}{\end{eqnarray*}}

\newtheorem{theorem}{Theorem}[section]
\newtheorem{proposition}[theorem]{Proposition}
\newtheorem{corollary}[theorem]{Corollary}
\newtheorem{lemma}[theorem]{Lemma}

\newtheorem{definition}[theorem]{Definition}
\theorembodyfont{\rm}

\newtheorem{remark}[theorem]{Remark}

\newtheorem{assumptions}[theorem]{Assumptions}

\newcommand{\qed}{\hfill $\Box$\smallskip}
\newcommand{\E}{\noindent{$\mathbb{E}$ \ }}

\def\P{\mathbb{P}}
\def\E{\mathbb{E}}

\def\e{\mathbb{E}}
\def\P{\mathbb{P}}

\def\cD{\mathcal{D}}

\def\cF{\mathcal{F}}

\def\cK{\mathcal{K}}
\def\cL{\mathcal{L}}

\def\cN{\mathcal{N}}
\def\cO{\mathcal{O}}

\def\cS{\mathcal{S}}

\def\Res{\text{Res}}

\def\txtd{{\textnormal{d}}}

\allowdisplaybreaks

\title{Amplitude equations for SPDEs driven by fractional additive noise with small Hurst parameter}
\author{Dirk Bl\"omker\thanks{Institut f\"ur Mathematik, Universit\"at Augsburg, Universit\"atsstra{\ss}e 12, 86135 Augsburg, Germany.~E-Mail: dirk.bloemker@math.uni-augsburg.de}~~~and~~Alexandra Neam\c tu\thanks{Department of Mathematics and Statistics, University of Konstanz, Universit\"atsstra{\ss}e~10, 78464 Konstanz, Germany. E-Mail: alexandra.neamtu@uni-konstanz.de} }
\begin{document}
\maketitle
\begin{abstract}
 We study stochastic partial differential equations (SPDEs) with potentially very rough fractional noise with Hurst parameter $H\in(0,1)$. Close to a change of stability measured with a small parameter $\varepsilon$, we rely on the natural separation of time-scales and establish a simplified description of the essential dynamics. We prove that up to an error term bounded by a power of $\varepsilon$ depending on the Hurst parameter we can approximate the solution of the SPDE in first order by an SDE, the so called amplitude equation, and in second order by a fast infinite dimensional Ornstein-Uhlenbeck process. To this aim we need to establish an explicit 
 averaging result for stochastic integrals driven by rough fractional noise for small Hurst parameters.
\end{abstract}

\section{Introduction}

The main goal of this work is to establish approximation results via amplitude equations for SPDEs perturbed by general additive fractional noise. More precisely, we consider SPDEs in  some Hilbert-space $X$ of the form
\begin{equation}\label{spde:i}
\begin{cases}
\txtd u = (L u+ \varepsilon^{2}\nu u +\mathcal{F}(u))~\txtd t + \varepsilon^{2H+1}~\txtd  W(t)\\
u(0)=u_0,
\end{cases}
\end{equation}
where $(W(t))_{t\geq 0}$ is a fractional Brownian motion in $X$ with Hurst index $H\in(0,1)$, $L$ is a self-adjoint non-positive 
operator with non-empty kernel $\mathcal{N}$ and
the nonlinearity $\mathcal{F}(u):=\mathcal{F}(u,u,u)$ is a trilinear cubic map, the prototype being $-u^3$. Furthermore the term $\varepsilon^2\nu$ for $\varepsilon,\nu>0$ can be viewed as a deterministic linear perturbation which shifts the spectrum of $L$. Here we will treat more general lower order linear terms later.

Several famous SPDEs such as  Ginzburg-Landau (also called Allen-Cahn eq.)
\begin{align*}
\txtd u = [\Delta u +\mu u - u^3]~\txtd t + \sigma \txtd W(t)
\end{align*}
or Swift-Hohenberg
\begin{align*}
    \txtd u =[ - (\Delta +1)^2 u +\mu u - u^3]~\txtd t +\sigma \txtd W(t)
\end{align*}
fall into our framework. The scaling of the parameters in \eqref{spde:i} is always in a way that linear (in)stability, 
the cubic nonlinearity and the noise are of the same order and have all an impact on the approximation.

In the  case $H=\frac{1}{2}$, where the forcing is a classical Wiener process we recover well-known results established in~\cite{BloemkerHairer1}. 
For example, for Swift-Hohenberg subject to periodic boundary conditions on the bounded domain 
$[0,2\pi]$ with $\mu=\nu\varepsilon^2$ and $\sigma=\varepsilon^2$ we obtain that 
\[
u(t) \approx \varepsilon a(\varepsilon^2t) + \varepsilon^2 Z(t)
\]
where $a$ lies in the two-dimensional dominant 
space of bifurcating patterns spanned by $\sin$ and $\cos$ and solves an SDE in that space. Moreover, $Z$ is a fast Ornstein-Uhlenbeck process given by the stochastic convolution orthogonal to the dominant space.

For $H\not=\frac12$ we have to take into account the scaling properties of the fractional Brownian motion
\begin{align*}
	& W(T\varepsilon^{-2}) \stackrel{\mbox{law}}{=} W(T)\varepsilon^{-2H} ~~\text{ and }	~~ \dot{W}(T\varepsilon^{-2}) \stackrel{\mbox{law}}{=} \varepsilon^{2-2H}\dot{W}(T).
\end{align*}
These indicates that the right scaling in~\eqref{spde:i} is indeed $\varepsilon^{2H+1}$, as this corresponds to rescaling to the slow time scale $ T=t\varepsilon^{2}$, which is introduced by the distance of order $\nu\varepsilon^2$ from the bifurcation. 
	
Similar to~\cite{BloemkerHairer1}, we rely on the natural separation of time-scales close to a change of stability and establish a simplified description of the essential dynamics proving an approximation result for the solutions of~\eqref{spde:i} via amplitude equations. However, in contrast to~\cite{BloemkerHairer1}, the error term of our approximation depends on the range of the Hurst parameter. More precisely,  we establish in Theorem~\ref{approx}
that 
for $H\geq \frac{1}{2}$ the approximation order is $3$, whereas for $H\in(0,\frac{1}{2})$ we obtain that the order is almost  
\[
	   1+2H+\frac{2H^2}{1-H}=\frac{1+H}{1-H} > 1.
\]
Note that this order is not only larger than order $1$, 
but also larger than $1+2H$, the scaling of the fractional noise.

There are several technical difficulties which we encounter when dealing with fractional noise. For instance even if the amplitude equation evolves on a slow-time scale, we need to establish an averaging result over a fast moving Ornstein-Uhlenbeck process, which is quite rough in time for small values of $H$.  We provide such a result in the proof of Lemma~\ref{pc:res} relying on Young-type estimates of convolutions combined with a suitable interpolation of the corresponding H\"older norms.

Amplitude equations are a well-known tool to 
study stochastic dynamics for SPDEs  close to a bifurcation 
especially for Brownian noise.
Apart from the already cited \cite{BloemkerHairer1}, which treats the Wiener case in a similar setting, there are numerous results.

In~\cite{BHP} the authors study large domains, where the dominating pattern is slowly modulated in space and the amplitude equation 
is thus an SPDE of Ginzburg-Landau type. See also \cite{BiDBSch:19} for an example of a fully unbounded domain.

Degenerate noise was studied in \cite{BloemkerMohammed} and several other publications. Here the noise is not acting directly on the dominating pattern, but through interaction of nonlinearity and
additive noise additional terms appear in the amplitude equation that  have the potential to stabilize the dynamics of the dominant modes. 

While all these references treated the case of additive Wiener noise, other types of noise were treated in the literature.
Multiplicative Wiener noise was studied by \cite{BloemkerFu}, 
while for additive fractional noise with $H>1/2$ certain results 
for Rayleigh-B\'enard convection are available in~\cite{BloemkerF}, 
but these results rely heavily on the smoothness of the noise and fail even in the case $H=1/2$.
Additive $\alpha$-stable L\'evy noise was studied recently in \cite{DBSY:22}, where the problem is not only the lack of smoothness,
but also the lack of moments.

Regarding qualitative properties of slow-fast systems with fractional noise we mention: homogenization results~\cite{GLi}, sample path estimates for additive fractional noise with $H>1/2$ \cite{EKN} and  averaging principles~\cite{HairerLi}. Here the slow variable is perturbed by multiplicative fractional noise, with Hurst index $H>1/2$ and the fast variable by a Brownian motion. These results rely on a novel application of the stochastic sewing lemma~\cite{KLe}. Based on this work averaging principles for slow-fast systems driven by independent fractional Brownian motions have been derived in~\cite{LiS}.
Here we follow a different more pathwise approach yielding for additive noise stronger error estimates.

Numerous extensions of our results are imaginable.
While including quadratic or higher order nonlinearities seem to be technical but possible,
unbounded or large domains as in \cite{BiDBSch:19} are a more challenging question.  
Moreover, it would be desirable to extend these results to SPDEs perturbed by rough nonlinear multiplicative noise. 
To this aim the explicit averaging we use does not seem to be available and one needs to generalize the results of~\cite{HairerLi} to SPDEs and derive averaging results for slow-fast SPDEs, where the slow variable is also perturbed by a random input which is rough in time. At the moment this seems out of reach but provides an exciting research perspective.

This work is structured as follows. In Section~\ref{s:fbm}
we collect important statements and properties of the fractional noise, for the convenience of the reader. Section~\ref{s:assumptions} contains the assumptions and a formal derivation of our main results which will become rigorous in Section~\ref{s:mainresults}. We conclude with two appendices providing H\"older-type estimates of (deterministic) convolutions and useful properties of stochastic convolutions with fractional noise. 


\section{Fractional Brownian Motion}
\label{s:fbm}


In this section we state important properties of the fractional Brownian motion which will be required later on. Further details can be looked up for example in~\cite{BiaginiHO, D}. We fix a complete probability space $(\Omega, \mathcal{F}, \mathbb{P})$ and use the abbreviation a.s.~for almost surely.

\begin{definition}
Let $H \in (0,1]$. A one-dimensional fractional Brownian motion (fBm) of Hurst index/parameter $H$ is a continuous centered Gaussian process $(\beta^H(t))_{t \geq 0}$ on $(\Omega,\mathcal{F},\mathbb{P})$ with covariance
\begin{align*}
\e[\beta^H(t) \beta^H(s)] = \frac{1}{2} \left( t^{2H} + s^{2H} - \left| t - s \right|^{2H} \right) \hspace{0.7cm} \text{for all } t,s \geq 0.
\end{align*}
\end{definition}

Note that for $H>1/2$ the covariance of the fBm satisfies 
\begin{align*}
\frac{1}{2} (t^{2H}  + s^{2H} - |t-s|^{2H} ) = H(2H-1) \int\limits_{0}^{t} \int\limits_{0}^{s} |v-u|^{2H-2} ~\diff v ~\diff u.
\end{align*}
We further observe that:
\begin{enumerate}
        \item [1)] for $H=1/2$ one obtains the Brownian motion;
        \item [2)] for $H=1$ then $\beta^{H}(t)= t \beta^{H}(1)$ a.s. for all $t\geq 0$. Due to this reason one always considers $H\in(0,1)$.
\end{enumerate}

The following result regarding the structure of the covariance of fBm holds true.

\begin{proposition} Let $H>1/2$. Then, the covariance of fBm has the integral representation
        \begin{align*}
        \e[\beta^H(t) \beta^H(s)]= \int\limits_{0}^{\min\{s,t\}} K(s,r) K(t,r)~\diff r~~\mbox{ for  } s,t\geq 0,
        \end{align*}
        where the integral kernel $K$ is given by
        \begin{align*}
        K(t,r)=c_{H}\int\limits_{r}^{t} \Big(\frac{u}{r} \Big)^{H-1/2}(u-r)^{H-3/2}~\diff u,
        \end{align*}
        for a positive constant $c_{H}$ depending exclusively on the Hurst parameter.
\end{proposition}

We remark that for suitable square integrable kernels, one obtains different stochastic processes, for instance the multi-fractional Brownian motion or the Rosenblatt process~\cite{CM}.

\begin{proposition}[{\bf Correlation of the increments}] 
Let $(\beta^{H}(t))_{t\geq 0}$ be a fBm of Hurst index $H\in(0,1)$. Then its increments are:
  \begin{itemize}
    \item [1)] positively correlated for $H>1/2$;
    \item [2)] independent for $H=1/2$;
    \item[3)] negatively correlated for $H<1/2$.
  \end{itemize}
Particularly, for $H>1/2$ fBm exhibits long-range dependence, i.e.
\[
\sum\limits_{n=1}^{\infty}\mathbb{E} [\beta^{H}(1) (\beta^{H}(n+1)-\beta^{H}(n))]=\infty ,
\]
whereas for $H<1/2$ 
\[
\sum\limits_{n=1}^{\infty}\mathbb{E} [\beta^{H}(1) (\beta^{H}(n+1)-\beta^{H}(n))]<\infty.
\]
\end{proposition}

\begin{proposition} Let $(\beta^{H}(t))_{t\geq 0}$ be a fBm of Hurst index $H\in(0,1)$. Then the following assertions hold true.
        \begin{itemize}
                \item [1)] [Self-similarity.] For $a\geq 0$
                \begin{align}\label{self:sim}
                (a^H \beta^H(t))_{t \geq 0} \overset{\emph{law}}{=} (\beta^H(at))_{t \geq 0},
                \end{align}
                i.e.~fBm is self-similar with Hurst index $H$.
                \item [2)] [Time inversion.] ~~$\Big(t^{2H}\beta^{H}(1/t) \Big)_{t>0} \overset{\emph{law}}{=}(\beta^{H}(t))_{t>0}. $
                \item[3)] [Stationarity of increments.] For all $h>0$
                \begin{align*}
                (\beta^H(t+h) - \beta^{H}(h))_{t \geq 0} \overset{\emph{law}}{=} (\beta^H(t))_{t \geq 0}.
                \end{align*}
                \item [4)] [Regularity of the increments.] The fBm has a version which is a.s.~H\"older continuous of exponent $0\leq \beta<H$.
                \item [5)] [Non-Markovianity.] For $H\neq 1/2$ the fBm is neither a semi-martingale, nor a Markov process.
        \end{itemize}     
\end{proposition}

For our aims, the scaling property~\eqref{self:sim} implies that
\[
 \beta^H(T\varepsilon^{-2}) \stackrel{\mbox{law}}{=} \beta^H(T)\varepsilon^{-2H} \qquad\text{and}\qquad 
 \dot{\beta}^{H}(T\varepsilon^{-2}) \stackrel{\mbox{law}}{=} \varepsilon^{2-2H}\dot{\beta}^{H}(T).
\]
In this work we consider an infinite-dimensional fractional Brownian motion with Hurst index $H\in(0,1)$ as defined below. Let $Q$ be a non-negative self-adjoint, trace-class operator on a separable Hilbert space $(X,\langle \cdot,\cdot \rangle)$. 
\begin{definition} 
Let $Q$ be a linear operator on $X$ of trace class. 
    An $X$-valued Gaussian process $(W^H(t)_{t\geq 0}$ on $(\Omega,\mathcal{F},\mathbb{P})$ is called a $Q$-fractional Brownian motion if
    \begin{itemize}
        \item [1)] $\e[W^H(t)]=0\quad\text{ for all } t\in\mathbb{R}_{+}$
        \item [2)] $\e \langle W^H(t),u \rangle 
        \langle W^H(s), v\rangle
        =\frac{1}{2}[t^{2H}+s^{2H}-|t-s|^{2H}]\langle Qu,v \rangle\quad\text{ for all } s,t\in\mathbb{R}_{+} \text{ and  } u,v \in X.$
    \end{itemize}
    In this case
    \begin{align*}
    W^H(t)=\sum\limits_{n=1}^{\infty} \sqrt{q_n}\beta^H_n(t) e_n,
    \end{align*}
    where $(e_n)_{n\geq 1}$ is an orthonormal basis of $X$ consisting of eigenvectors of $Q$, $(q_n)_{n\geq 1}$ are the corresponding non-negative eigenvalues, i.e. $Qe_n=q_n e_n$ and $(\beta^H_n(t))_{t\geq 0}$ is a sequence of independent standard fractional Brownian motions with Hurst index $H$. 
\end{definition}

%
\section{Notations and assumptions}\label{s:assumptions}
%

We introduce suitable notations and assumptions which will be used throughout this manuscript. If not further stated, $X$ stands for a separable Hilbert space, $\|\cdot\|$ and $\langle\cdot,\cdot\rangle$ always stand for the norm, respectively for the scalar product in $X$. In the following we list classical assumptions on the linear operators $L$ and $A$, the cubic nonlinear term $\cF$ and the noise, similar to~\cite{BloemkerHairer1}.

\begin{assumptions}(Differential operator $L$)\label{ass:L}
	The operator $L$ has a compact resolvent and generates an analytic $C_0$-semigoup $\{e^{tL}\}_{t\geq 0}$ in $X$. We denote its domain by $\cD(L)$, set $\cN:=\text{Ker}(L)$ and assume that this is finite dimensional. 
	Moreover, we denote with $P_c$ the orthogonal projection onto $\cN$, set $P_s:=\text{Id}-P_c$ and assume for simplicity that $P_s$ and $P_c$ commute with $L$ and therefore with the semigroup $(e^{tL})_{t\geq 0}$. 
	Furthermore we define $\cS=P_sX$ and obtain an orthogonal decomposition of $X$ as $X=\cN\oplus \cS$. 
	We assume that there exist constants $M\geq 1$ and $\mu>0$ such that
	\begin{align}\label{stability:s}
	\|e^{tL}P_s\|_{\cL(X,X)} \leq M e^{-t\mu} \quad\text{ for all } t\geq 0.
	\end{align}
We define the fractional power spaces $X^{\alpha}:=\cD((\text{Id}-L)^{\alpha})$ for $\alpha\geq 0$.
\end{assumptions}
We will use the following well-known estimates for analytic semigroups and fractional power spaces.
\begin{itemize}
	\item [1)] For all $\alpha\in[0,1)$ there exists a constant $M\geq 1$ such that
	\begin{align*}
	\|e^{tL}\|_{\cL(X^{-\alpha},X)}\leq M (1+t^{-\alpha})\quad\text{ for all } t>0.
	\end{align*}
	\item [2)] Combining this with~\eqref{stability:s} we obtain the existence of a constant $\tilde{M}\geq 1$ such that
	\begin{align}\label{fr:power:s}
	\|P_s e^{tL}\|_{\cL(X^{-\alpha},X)} \leq \tilde{M} (1+t^{-\alpha})e^{-t\tilde{\mu}}\quad \text{ for all } t>0,
	\end{align}
where $0<\tilde{\mu}<\mu$.
\end{itemize}
\begin{assumptions}\label{ass:A}(Bounded perturbation $A$)
	We  assume that $A:D(L)\to X$ and that $A$ is a bounded linear operator from $X$ to $X^{-\alpha}$ for some $\alpha\in[0,1)$.
\end{assumptions}
Note that $A_c=P_cA$ is a bounded linear operator on $X$.
Here we use the standard shorthand notations to refer to $A$ and $\cF$ on $\cN$ respectively $\cS$, i.e. $A_{c/s}=P_{c/s}A$, $\cF(u)=\cF(u,u,u)$, $\cF_{c/s}=P_{c/s}\cF$ and so on.
\begin{assumptions}(Nonlinear term)\label{ass:F}
	The function $\cF:X^3\to X^{-\alpha}$ for $\alpha\in[0,1)$ from Assumption \ref{ass:A} is continuous, trilinear and symmetric. 
	 Furthermore, we impose the following sign conditions on $\cN$
	\begin{align}
	\langle \cF_c(v_c),v_c \rangle <0 \quad &\text{ for all } v_c\in\cN\setminus\{0\},\label{f:sign}
	\\
	\langle \cF_c(v_c,v_c,w_c),w_c\rangle < 0 \quad &\text{ for all  } v_c,w_c \in \cN\setminus\{0\}\label{f1}
	\end{align}
	and assume that for some small $\eta$ we can find a constant $C_\eta>0$ such that 
		\begin{align}\label{f:stable:n}
	    \langle \cF_c(\phi+v_c),v_c\rangle  \leq C_\eta \|\phi\|^4 - \eta \|v_c\|^4, \text{ for all } v_c,\phi\in\cN.
	\end{align}
\end{assumptions}

\begin{assumptions}(Noise)\label{ass:W}
	The noise term is a trace-class fractional Brownian motion with Hurst index $H\in(0,1)$ with values in $X$.
\end{assumptions}

Throughout this manuscript we use the $\cO$ notation in the following way.
\begin{definition}
    We say that a term $F_{\varepsilon}=\cO(f_{\varepsilon})$ if and only if there exist positive $\varepsilon$-independent constants $C$ and $\varepsilon_0$ such that $|F_{\varepsilon}|\leq C f_{\varepsilon}$ for all $\varepsilon\in(0,\varepsilon_0]$.
\end{definition} 


\begin{assumptions} (Initial data)\label{incond}
    We assume that $u_0=\cO(\varepsilon^{1-\kappa})$ and $P_su_0=\cO(\varepsilon^{2H+1-\kappa})$ for a small $\varepsilon>0$ and a small $\kappa>0$.
\end{assumptions}



\paragraph{Formal derivation of the main result.}
 We make the usual ansatz known from the white noise case, 
where we suppose close to a change of stability small solutions 
evolving on a slow time-scale on the dominant modes, 
while all the other modes are subject to fast damping and small. 
Since we deal with fractional noise we make the ansatz
\begin{align}\label{ansatz}
u(t)=\varepsilon a(\varepsilon^{2}t) + \varepsilon^{2H+1}\psi_s(t) + \mathcal{O}(\varepsilon^{\gamma}),
\end{align}
where $a\in \cN$, $\psi_s\in \cS$ and the order $\gamma$ of the error depends on the range of the Hurst parameter. We will show that for $H\geq \frac{1}{2}$ the approximation is almost of order $\gamma=3$, whereas for $H\in(0,\frac{1}{2})$ we will obtain almost the order $\gamma=2H+1+\frac{2H^2}{1-H}=\frac{1+H}{1-H}$. 

Rescaling to the slow time-scale $T=\varepsilon^{2} t$, plugging the ansatz \eqref{ansatz} into~\eqref{spde:i}, projecting onto $\cN$ and collecting the terms of larges order $\varepsilon^{3}$ yields the amplitude equation on the slow time-scale
\begin{align*}
\partial_{T} a = [\nu a - P_{c} a^{3} ] ~\txtd T +  \partial_{T}b(T),
\end{align*}
consequently
\begin{align*}
a(T) = a_{0} +\int\limits_{0}^{T}(\nu  a(s) -P_{c}(a(s) )^{3})~\txtd s + b(T),
\end{align*}
where 
\[
b(T):=\varepsilon^{2H}P_{c}W(\varepsilon^{-2}T)
\] 
is a rescaled fractional Brownian motion in $\cN$ with distribution independent of $\varepsilon$. Thus the law of 
$a$ is also independent of $\varepsilon$. Projecting onto $\cS$ and collecting the terms of order $\varepsilon^{2H+1}$ results in
\begin{align*}
\txtd\psi_s = L \psi_s ~\txtd t + P_{s} ~\txtd W(t),
\end{align*} 
so $\psi_s$ is a fractional Ornstein-Uhlenbeck process given by
\begin{align}\label{ou}
\psi_s(t) = e^{tL}\psi_s(0) + \int_0^t P_s e^{L(t-s)}~\txtd W(s).
\end{align}

 The main goal of the next section is to provide a rigorous proof of~\eqref{ansatz}. To this aim we need averaging results for the error term that contains products of the fast moving Ornstein-Uhlenbeck process $\psi_s$
 and the amplitude $a$ which 
 evolves on the slow-time scale. 
 Note that the fast Ornstein-Uhlenbeck process can be very rough in time when the values of $H$ are close to zero.

%
\section{Statement of the problem and main results}\label{s:mainresults}
%

Let $\varepsilon>0$ be small. Under Assumptions~\ref{ass:L}--\ref{incond} we consider the SPDE
\begin{equation}\label{spde}
\begin{cases}
\txtd u = (L u+ \varepsilon^{2}A u +\cF(u) )~\txtd t + \varepsilon^{2H+1}~\txtd  W(t)\\
u(0)=u_0,
\end{cases}
\end{equation}
where $W:=W^{H}$ is a fractional Brownian motion with Hurst index $H\in(0,1)$. 
\begin{definition}{\bf (Mild solution)}
	Let $t^{\star}$ be a positive stopping time. A stochastic process $u:[0,t^\star)\to X$ is called mild solution for~\eqref{spde} if $\mathbb{P}$-a.s.
	\begin{align}\label{mild:sol}
	u(t) = e^{tL} u_0 +\int\limits_0^t e^{(t-\tau)L} [\varepsilon^2 A u (\tau) +\cF(u(\tau))]~\txtd \tau  +\varepsilon^{2H +1}\int\limits_0^t e^{(t-\tau)L}~\txtd W(\tau) ~\quad t\in(0,t^{\star}).
	\end{align}
	Moreover, the solution is maximal so that either $t^{\star}=\infty$ or $\|u(t)\|_{X}\to\infty$ for $t\nearrow t^{\star}$.
\end{definition}
The stochastic integral
\begin{align*}
W_L(t): =\int\limits_0^t e^{(t-\tau)L}~\txtd W(\tau)
\end{align*}
is called stochastic convolution. For its precise construction we refer to~\cite{DuncanMaslowski} for $H\in(\frac{1}{2},1)$ and~\cite{DuncanMaslowski2} for $H\in(0,\frac{1}{2})$. Under suitable assumptions on the coefficients there exists a global-in-time solution for~\eqref{spde}, see~\cite{DuncanMaslowski, GubinelliTindel} and the references specified therein. For our aims, local-in-time existence of mild solutions suffices, which is standard due to our assumptions on the coefficients. The main tool in the proof is a pathwise fixed-point argument, but we do not give any further details here.

Since the projections commute with the semigroup, we have
\begin{align*}
P_s [W_L(t)] =\int\limits_0^t e^{(t-\tau)L}~\txtd P_sW(\tau) \quad \text{ and } \quad P_c[W_L(t)]=P_cW(t).
\end{align*}
Applying $P_s$ and $P_c$ to~\eqref{mild:sol} we obtain 
\begin{align*}
u_s(t):=P_s u(t) = e^{tL}P_su_0 + \int\limits_0^t e^{(t-\tau)L} [\varepsilon^2 A_s u(\tau) +\cF_s(u(\tau)) ]~\txtd\tau + \varepsilon^{2H+1} \int\limits_0^t e^{(t-\tau)L}~\txtd P_s W(\tau)
\end{align*}
and 
\begin{align*}
u_c(t):=P_c u(t) =P_c u_0 +\int\limits_0^t [\varepsilon^2 A_c u(\tau) +\cF_c(u(\tau)) ]~\txtd\tau + \varepsilon^{2H+1}P_c W(t).
\end{align*}
\begin{definition}
	We call $u_s(t)=P_s u(t)\in\cS$ fast modes, since they are subject to a deterministic exponential decay on a time-scale of order $\cO(1)$. Moreover $u_c(t)=P_cu(t)\in\cN$ are referred to as slow modes, since they change only on the slow time scale $T=\varepsilon^2 t$.
\end{definition}

\begin{remark}\label{k:rem}
Note that the influence of the noise on the slow and fast time scale is different. More precisely, due to the scaling properties of the fractional Brownian motion~\eqref{self:sim}, we have for every small $\kappa>0$
\begin{align*}
\sup\limits_{t\in[0,T_0\varepsilon^{-2}]}\|P_c W(t)\| \stackrel{\mbox{law}}{=} \varepsilon^{-2H} \sup\limits_{T\in[0,T_0]} \|P_cW(T)\| =\mathcal{O}(\varepsilon^{-2H-\kappa})
 \end{align*}
 with probability almost $1$,
 whereas, as we show in Appendix~\ref{stoch:conv}, with high probability
 \begin{align*}
    \sup\limits_{t\in[0,T_0\varepsilon^{-2}]} \|P_sW_L(t)\|=\cO(\varepsilon^{-\kappa}).
 \end{align*}
 This means that on $u_c$ the noise acts like a fractional Brownian motion in $\cN$, whereas on the fast component $u_s$ we have an infinite-dimensional fractional Ornstein-Uhlenbeck process. The small exponent $\kappa>0$ is only inserted for simplicity in order to have the probability close to $1$ for small $\varepsilon$. 
\end{remark}

The main goal of this section is to justify~\eqref{ansatz}  and to control the approximation error. To this aim we firstly introduce the stochastic process $a\in C([0,T];\cN)$ which solves the amplitude equation
\begin{align}\label{amplitude:eq}
\partial_T a(T)= A_c a(T) + \cF_c(a(T)) + \partial_T b(T)
\end{align}
with initial condition $a(0)=\varepsilon^{-1}P_cu_0$
where $b(T)=\varepsilon^{2H}P_c(W(T\varepsilon^{-2}))$ is a finite-dimensional fractional Brownian motion on $\cN$ independent of $\varepsilon$ in law due to the scaling properties of $W$. 
We further set $\psi_c(t):=a(\varepsilon^2 t)$. We also define the fractional Ornstein-Uhlenbeck process $\psi_s$ as 
\begin{align}\label{psi:s}
\psi_s(t)=e^{tL}\psi_s(0) +P_sW_L(t),
\end{align}
where $\psi_s(0)=\varepsilon^{-2H-1}P_su_0$. We define our approximation of~\eqref{spde} by
\begin{align}\label{psi}
\psi(t):=\varepsilon \psi_c(t) +\varepsilon^{2H+1} \psi_s(t). 
\end{align}
In order to prove the error estimate, we need to control the following residual
\begin{align}\label{residual}
\Res(\psi(t)):=-\psi(t) + e^{tL}\psi(0) +\int\limits_0^t e^{(t-\tau)L} [\varepsilon^2 A(\psi(\tau)) +\cF(\psi(\tau)) ]~\txtd \tau + \varepsilon^{2H+1}W_L(t),
\end{align}
which measures the quality of the approximation $\psi$.
In conclusion we have to  establish the following results:
\begin{itemize}
	\item attractivity;
	\item bounds on the residual;
	\item control the approximation order.
\end{itemize}
%
%
\paragraph{Attractivity.} 
If the initial condition $u_0$ is of order $\varepsilon$ then the attractivity result shows that 
$u_c(t_{\varepsilon})$ is of order $\varepsilon$ 
and $u_s(t_{\varepsilon})$ is of order $\varepsilon^{2H+1}$ for all times $t_\varepsilon$ of order $\ln(\varepsilon^{-2H})$. This justifies at time $t_{\varepsilon}$ 
the scaling of the formal ansatz 
we use in the approximation in \eqref{psi}.

First we verify that $u$ stays uniformly small up to a logarithmic time if we take small initial data.
\begin{lemma}\label{lemma:attractivity}
	Let Assumptions~\ref{ass:L}--\ref{incond} hold. 
	For all times $t_{\varepsilon}=\cO(\ln(\varepsilon^{-2H}))$, 
	all constants $\delta>0$, 
	all $\kappa\in[0,1)$,
	and $\varepsilon\in(0,1]$ sufficiently small
    suppose
	\begin{align*}
	\sup\limits_{t\in[0,t_{\varepsilon}]} \|W_L(t)\|\leq \varepsilon^{-2H-\kappa} \quad \text{ and } \quad \|u_0\|\leq \delta\varepsilon^{1-\kappa}.
	\end{align*}
	Then we obtain for any $D > M\delta +2$ and for a maximal mild solution $u$
	\begin{align*}
	\sup\limits_{t\in[0,t_{\varepsilon}]} \|u(t)\|\leq D\varepsilon^{1-\kappa}.
	\end{align*}
\end{lemma}
\begin{proof}
	Define the stopping time $\tau^{\star}_{\varepsilon}:=\inf\{\tau>0 : \|u(\tau)\|>D\varepsilon^{1-\kappa} \}$. 
	Hence for $t<\min\{\tau^{\star}_{\varepsilon},t_{\varepsilon}\}$, which is by definition smaller than the maximal time of existence for the mild solution, 
	we obtain
	\begin{align*}
	\|u(t)\|
	&\leq M \|u_0\| +M \int\limits_{0}^{t} (1+(t-\tau)^{-\alpha} ) \|\varepsilon^2 A u(\tau) + \cF (u(\tau))\|_{-\alpha}~\txtd \tau + \varepsilon^{2H+1}\|W_L(t)\|\\
	& \leq [M\delta +1]\varepsilon^{1-\kappa} + M  D^2 \varepsilon^{3-3\kappa} 
	\int\limits_0^{t_{\varepsilon}} (1+\tau^{-\alpha})~\txtd \tau \\
	& \leq [M \delta + 2]\varepsilon^{1-\kappa}\\
	& < D\varepsilon^{1-\kappa},
	\end{align*}
	for $\varepsilon$ small enough. 
	\qed\end{proof}

\begin{corollary}
Under the assumptions of Lemma \ref{lemma:attractivity}
we have 
\[	\mathbb{P}\Big(\sup\limits_{t\in[0,t_{\varepsilon}]} \|u(t)\|> D\varepsilon^{1-\kappa}\Big)
\leq
\mathbb{P}\Big(	
	\sup\limits_{t\in[0,t_{\varepsilon}]} \|W_L(t)\|> \varepsilon^{-2H-\kappa}\Big) 
	+ \mathbb{P}\Big(\|u_0\|> \delta\varepsilon^{1-\kappa}\Big).
\]
\end{corollary}
Note that the statement in the corollary holds for $\kappa \geq0$, but in order to obtain small probabilities
vanishing in the limit $\varepsilon\to0$ 
we need $\kappa>0$. 

\begin{remark}[Stochastic Convolution] \label{rem:SC}
In order to obtain bounds for the probability in the previous corollary, we need first a bound on 
\[
\P(\sup\limits_{t\in[0,t_{\varepsilon}]} \|P_cW_L(t)\| > \tfrac12\varepsilon^{-2H-\kappa}) 
 = \P(\sup\limits_{t\in[0,1]} \|P_cW_L(t)\| > \tfrac12\varepsilon^{-2H-\kappa}t_{\varepsilon}^{-H}) 
\]
which is small, as $P_cW_L$ is a Gaussian random variable 
in $C^0([0,1],X)$.

Secondly, one can use the factorization method to bound the probability $\P(\sup\limits_{t\in[0,t_{\varepsilon}]} \|P_sW_L(t)\| > \tfrac12\varepsilon^{-2H-\kappa})$.  See~\cite{BloemkerHairer1} for the analogous statement for $H=\frac{1}{2}$. In this case already the probability $\P(\sup\limits_{t\in[0,t_{\varepsilon}]} \|W_L(t)\| > \varepsilon^{-\kappa})$ 
is bounded by terms of order $\cO(\varepsilon^p)$ for all $p\geq 1$. For further details regarding fractional stochastic convolutions and the smallness of probability, consult Appendix~\ref{stoch:conv}. 

Note that the bound on $\|P_sW_L(t_\varepsilon)\|$ is not surprising, since the law of $P_sW_L(t_{\varepsilon})$ converges to a unique limiting measure for $t_{\varepsilon}\to\infty$, see~\cite[Proposition 3.4]{DuncanMaslowski} for $H>1/2$ and \cite[Proposition 11.13]{DuncanMaslowski2} for $H<\frac{1}{2}$.
\end{remark}

\begin{theorem}{\em (Attractivity)}
	Let Assumptions~\ref{ass:L}--\ref{incond} hold, 
	fix $\kappa\in[0,1-H]$, constants $C_s>0$, $\delta>0$
	and a time  $t_{\varepsilon}=\frac{1}{\mu}\ln(\varepsilon^{-2H})$,
and rewrite a maximal mild solution $u$ of~\eqref{spde} 
	at time $t_{\varepsilon}$ as
	\begin{align*}
	u(t_{\varepsilon})=\varepsilon a_{\varepsilon}  + \varepsilon^{2H+1} R_{\varepsilon},
	\end{align*}
	with $a_{\varepsilon}\in\cN$ and $R_{\varepsilon}\in \cS$.
	
	Then there is a constant $C>0$ such that the following hold. 
	If for a sufficiently small $\varepsilon\in(0,1]$ 
	\begin{align*}
	\|u_0\|\leq \delta \varepsilon^{1-\kappa}, \quad \sup\limits_{t\in[0,t_{\varepsilon}]} \|W_L(t)\|
	\leq \varepsilon^{-2H-\kappa} \quad \text{ and } \quad \|P_sW_L(t_{\varepsilon})\|\leq C_s\varepsilon^{-\kappa},
	\end{align*}
	then we have
	\begin{align*}
	\|a_{\varepsilon}\| \leq C \varepsilon^{-\kappa}
	\quad \text{ and }\quad  
	\|R_{\varepsilon}\|\leq C \varepsilon^{-\kappa}.
	\end{align*}
\end{theorem}
\begin{proof} We set $a_{\varepsilon}=\varepsilon^{-1}P_c(u(t_{\varepsilon}))$ and 
$R_{\varepsilon}=\varepsilon^{-(2H+1)}P_su(t_{\varepsilon})$ and recall that by Lemma~\ref{lemma:attractivity} we have
$\|u(t_\varepsilon)\| \leq D \varepsilon^{1-\kappa}$.
Thus we only need to establish the bound on the fast modes, since $| a_{\varepsilon}|=\varepsilon^{-1}|P_c(u(t_{\varepsilon}))|$. 

For $P_su$ we derive
	\begin{align*}
	\|P_s u(t_{\varepsilon})\|&\leq M e^{-\mu t_{\varepsilon}}\|u_0\| + \varepsilon^{2H+1}\|P_sW_L(t_{\varepsilon})\| \\
	&  +M \int\limits_0^{t_{\varepsilon}} (1+(t_{\varepsilon}-\tau))^{-\alpha} e^{-\mu (t_{\varepsilon}-\tau) } \|\varepsilon^2 A(u(\tau))+\cF(u(\tau)) \|~\txtd\tau 
	\end{align*}
	Since $\tau\leq t_{\varepsilon}$, $t_{\varepsilon}=\frac{1}{\mu}\ln(\varepsilon^{-2H})$ and $\|u(\tau)\|\leq D\varepsilon^{1-\kappa}$, as established in Lemma~\ref{lemma:attractivity}, we  conclude
	\begin{align*}
	\|P_s u(t_{\varepsilon})\| 
	\leq M\delta \varepsilon^{1-\kappa+2H} 
	+ C_s \varepsilon^{2H+1-\kappa} 
	+ MCD \varepsilon^{3-3\kappa}\int\limits_0^{\infty} (1+\tau^{-\alpha})e^{-\mu\tau}~\txtd\tau.
	\end{align*}
	This proves that $R_\varepsilon=  \varepsilon^{-2H-1} P_su(t_\varepsilon) =\cO(\varepsilon^{-\kappa})$, as $\kappa\leq 1-H$.	\qed\end{proof}

\begin{definition}[Set of large probability]
\label{def:Omega}
Let $\kappa> 0$ small enough and fix a time $T_0>0$. Furthermore, if $H>1/2$ let $\beta$ such that $1/2<\beta<H$ and if $H\leq 1/2$ set $\beta:=H-\kappa$.  
We introduce a subset $\tilde\Omega_\varepsilon$ of $\Omega$  as follows:
\begin{align*}
\tilde\Omega_\varepsilon:=
\Big\{ \omega\in\Omega & : \sup\limits_{t\in[0,T_0\varepsilon^{-2}]} \|P_sW_L(t)\|\leq   \varepsilon^{-\kappa}, 
\| W(\varepsilon^{-2}\cdot)\|_{C^{\beta}([0,T_0];X)}\leq \varepsilon^{-2H-\kappa},\\& ~~  \|\psi_s(0)\|\leq \varepsilon^{-\kappa},~~  \|\psi_c(0)\|\leq \varepsilon^{-\kappa}\Big\}.
\end{align*} 
\end{definition}
Note that the smallness of $\kappa$ depends also on $H$.

\begin{remark}
Obviously the probability of $\tilde\Omega_\varepsilon$ is large. See also the discussion in Remark \ref{rem:SC}. 
We expect that  for all $p>1$ there is a constant $C_p>0$ such that 
\[
\mathbb{P}(\tilde\Omega_\varepsilon) \geq 1-C_p \varepsilon^p.
\]
\end{remark}
For the first term due to the factorization lemma and Chebyshev's inequality, we have according to Appendix~\ref{stoch:conv} that 
\begin{align*}
\mathbb{P}\Big(\sup\limits_{t\in[0,T_0\varepsilon^{-2}]} \|P_sW_L(t)\|>\varepsilon^{-\kappa} \Big)
\leq \varepsilon^{\kappa p} \mathbb{E}\sup\limits_{t\in[0,T_0\varepsilon^{-2}]} \|P_sW_L(t)\|^p
\leq C \varepsilon^{\kappa p}
T_0\varepsilon^{-2}.
\end{align*}
Since $p$ can be taken arbitrarily large, we conclude that $\kappa>0$ is necessary in order to make this probability arbitrarily small.
For more details on stochastic convolutions with fractional noise, see Appendix~\ref{stoch:conv}.

To bound the probability of the H\"older term in the definition of $\tilde\Omega_\varepsilon$ we use Chebyshev's inequality and the self-similarity of the fractional Brownian motion to obtain that 
\[
\mathbb{P}\Big(\|W(\varepsilon^{-2}\cdot)\|_{C^\beta([0,T_0];X)} > \varepsilon^{-2H-\kappa}  \Big) 
= \mathbb{P}\Big(\|W\|_{C^\beta([0,T_0];X)}  > \varepsilon^{-\kappa} \Big) 
\leq \varepsilon^{p\kappa} \mathbb{E}\|W\|_{C^\beta([0,T_0];X)}^p 
\] 
which is bounded by any power in $\varepsilon$, as $W$ is a Gaussian in $C^\beta([0,T_0];X)$.
This is the reason why we work with $X$-valued trace-class noise.

In the following we establish a bound on the solution 
$a$ of the amplitude equation
\begin{align}\label{a:amplitude}
\partial_T a(T)= A_c a(T) + \cF_c(a(T)) + \partial_T b(T),
\end{align}
where $b(T)=\varepsilon^{2H}P_c(W(T\varepsilon^{-2}))$ is a rescaled fractional Brownian motion.
\begin{lemma}{\em (Bound on $a$)}\label{lemma:a} Let Assumptions~\ref{ass:L}--\ref{incond} hold and consider the set $\tilde\Omega_\varepsilon$ from Definition \ref{def:Omega}.
Then there exists a positive constant $C_a$  such that
	\begin{align}\label{bound:a}
	\tilde\Omega_\varepsilon\subset \Big\{ \omega\in\Omega~: \sup\limits_{T\in[0,T_0]}\|a(T)\|\leq C_a\varepsilon^{-2\kappa},~ \|a\|_{C^\beta([0,T_0];\cN)} \leq C_a\varepsilon^{-6\kappa}\Big\}.
	\end{align}
\end{lemma}
\begin{proof}
Note that on the finite-dimensional space $\mathcal{N}$ 
all norms are equivalent, thus we use the one induced by $X$ for simplicity.
 This proof is fairly standard, we only give the main ideas. Subtracting the finite-dimensional stochastic convolution 
\begin{align*}
    \tilde{b}(T)= \int\limits_0^T e^{-(T-S)}~\txtd b(S),
\end{align*}
i.e. setting $c:=a-\tilde{b}$ one obtains the random ODE on the slow-time scale $T>0$
\begin{align*}
\partial_T c = A_c (c+\tilde{b}) + \cF_c(c+\tilde{b}) + \tilde{b}.
\end{align*}
Now we take the inner product with $c(T)$ and use that $A_c\in\cL(\cN)$ and~\eqref{f:stable:n} to obtain
\begin{align*}
\frac{1}{2}\partial_T \|c\|^2 &=\langle A_c(c+\tilde{b}),c\rangle + \langle \cF_c(c+\tilde{b}),c\rangle +\langle\tilde{b},c\rangle\nonumber\\
& \leq K \|c\|^2 + K \|\tilde{b}\|^2 + \langle\cF_c (c+\tilde{b}),c \rangle\nonumber\\
& \leq K \|c\|^2 + K \|\tilde{b}\|^2 + C_\eta \|\tilde{b}\|^4 - \eta \|c\|^4\nonumber\\
& \leq K \|c\|^2 + K (1+\|\tilde{b}\|^2)^2,
\end{align*}
for positive constants all denoted by $K$. 
Now, Gronwall's inequality entails
\begin{align*}
    \|c(T)\|^2 \leq e ^{K T} \|c(0)\|^2 + K\int_0^T e ^{K(T-\tau)}(1+\|\tilde{b}(\tau)\|^2)^2~\txtd \tau,
\end{align*}
consequently
\[
\|c(T)\|^2 \leq K_T\left[\|c(0)\|^2+ 1 + \|b\|^4_{C^0([0,T_0],\mathcal{N})} \right]. 
\]
This means that $\|c(T)\|\leq K_T \varepsilon^{-2\kappa}$, therefore
using that $\|a(T)\|\leq \|c(T)\| +\| \tilde{b}(T)\|$ and taking the supremum over $[0,T_0]$ proves the first bound in~\eqref{bound:a}. Alternatively, one can use~\eqref{f:sign} to obtain a differential inequality of the form
\begin{align*}
    \partial_T \|c\|^2 \leq - \tilde{C} \|c\|^4 + \tilde{C} (1+\|\tilde{b}\|^2)^2
\end{align*}
and apply a comparison argument for ODEs to bound $c$.
Similar arguments were used for $H=\frac{1}{2}$ in \cite[Lemmas 4.3, 4.5]{BloemkerHairer1}. 

In order to estimate the $\beta$-H\"older norm of the amplitude equation we directly have from~\eqref{a:amplitude} that
\begin{align*}
    a(T)-a(S) = A_c\int_S^T a(\tau)~\txtd\tau +\int_S^T \cF_c (a(\tau))~\txtd \tau + b(T) - b(S).
\end{align*}
Using that $A_c\in\cL(\cN)$, $\sup\limits_{T\in[0,T_0]}\|a(T)\|\leq C_a\varepsilon^{-2\kappa}$ on $\tilde\Omega_\varepsilon$, we obtain due to the cubic nonlinear term that
\begin{align*}
\|a(T)-a(S) \|&\leq \int_S^T \|a(\tau)\|~\txtd \tau + \int_S^T \|\cF_c (a(\tau))\|~\txtd \tau  + \|b(T)-b(S)\|\\
& \leq C_a (\varepsilon^{-2\kappa} +\varepsilon^{-6\kappa}) (T-S) + (T-S)^\beta\|b\|_{C^\beta([0,T_0];\cN)}. 
\end{align*}
This proves the second statement of~\eqref{bound:a}.

\qed	\end{proof}

Now we establish a bound on the projection of the approximation onto $\mathcal{N}^\perp$.

\begin{lemma} {\em (Bound on $\psi_s$)}\label{lemma:psi_s} 
	Let Assumptions~\ref{ass:L}--\ref{incond} hold  and consider the set $\tilde\Omega_\varepsilon$ from Definition \ref{def:Omega}.  Then  there exists a constant $C_s>0$ such that
	\begin{align}\label{e:psi_s}
	\tilde\Omega_\varepsilon \subset \Big\{\omega\in\Omega ~:\sup\limits_{t\in[0,T_0\varepsilon^{-2}]} \|\psi_s(t)\|\leq C_s \varepsilon^{-\kappa}\Big\}.
	\end{align}
\end{lemma}

	\begin{proof}
		Using the definition of $\psi_s$ given in~\eqref{psi:s} together with the fact that $P_su_0=\cO(\varepsilon^{2H+1-\kappa})$ and the exponential stability of the semigroup on $P_s$ assumed in~\eqref{fr:power:s}, we get
		\begin{align*}
		    \|\psi_s(t)\| 
		    & \leq M \|\psi_s(0)\| + \|P_sW_L(t)\|\\
		    &= M \varepsilon^{-2H-1} \|P_su_0\| + \|P_sW_L(t)\|\\
		    & \leq C \varepsilon^{-\kappa}
		\end{align*}
		on $\tilde\Omega_\varepsilon$, which proves the assertion.
		\qed\end{proof}
		
\paragraph{Residual}
We now focus on the residual defined in \eqref{residual}. 
\begin{lemma}\label{ps:res}
	Let Assumptions~\ref{ass:L}--\ref{incond} hold and consider the set $\tilde\Omega_\varepsilon$ from Definition \ref{def:Omega}. Then  there exists a constant $C_{\text{\em s,Res}}>0$ such that for sufficiently small $\varepsilon>0$	
	\begin{align*}
	\tilde\Omega_\varepsilon\subset \Big\{\omega\in\Omega : \sup\limits_{t\in[0,T_0\varepsilon^{-2}] } \|P_s\text{\em Res}(\psi(t))\| \leq C_{\text{\em s,Res}} \varepsilon^{3-6\kappa}\Big\}.
	\end{align*}
\end{lemma}
\begin{proof}
Taking the stable projection on~\eqref{residual} and using the definition of $\psi_s$ results in
\begin{align}\label{res:s}
P_s\Res(\psi(t))=\int\limits_0^t e^{(t-\tau)L} P_s (\varepsilon^2 A\psi(\tau)+\cF(\psi(\tau)))~\txtd \tau.
\end{align}
	Using Assumptions~\ref{ass:A} and \ref{ass:F} we infer
	\begin{align*}
	\|P_s\Res(\psi(t))\| &\leq M \int\limits_0^t (1+(t-\tau))^{-\alpha} e^{-\mu (t-\tau)} (\varepsilon^2 \|\psi(\tau)\| +\|\psi(\tau)\|^3 )~\txtd \tau\\
	&\leq M \int\limits_0^t (1+(t-\tau))^{-\alpha} e^{-\mu (t-\tau)} ~\txtd \tau~  \sup\limits_{\tau\in[0,T_0\varepsilon^{-2}]} (\varepsilon^2 \|\psi(\tau)\| +\|\psi(\tau) \|^3 )\\
	& \leq  M \int\limits_0^{\infty} (1+\tau^{-\alpha})e^{-\mu \tau}~\txtd \tau  ~  \sup\limits_{\tau\in[0,T_0\varepsilon^{-2}]} (\varepsilon^2 \|\psi(\tau)\| +\|\psi(\tau) \|^3 ).
	\end{align*}
Recall that by Lemmas \ref{lemma:a} and \ref{lemma:psi_s}   
\[
\psi(t)= 
\varepsilon a(\varepsilon^2t) + \varepsilon^{2H+1}\psi_s(t) 
= \mathcal{O}(\varepsilon^{1-2\kappa})
\quad\text{ on }\tilde\Omega_\varepsilon, 	
\]	
which proves the statement.
	\qed\end{proof}

However, the estimate of $P_c\Res(\psi)$ is significantly more involved. Such estimates rely on averaging results, which are highly challenging to obtain for fractional noise~\cite{HairerLi,LiS}. Here we rely on an explicit pathwise approach possible for additive noise.

For our aims, we firstly prove some auxiliary results.
\begin{lemma}\label{lemma:rel:young} The following relation
\begin{align}\label{pswl}
 \int\limits_0^t P_sW_L(\tau\varepsilon^{-2})~\txtd \tau &=  \int\limits_0^t e^{(t-s)L\varepsilon^{-2}} P_sW(s\varepsilon^{-2})~\txtd s
 \end{align}
is valid.
\end{lemma}
\begin{proof}
Using the definition of $W_L$, the substitution $\tau\mapsto \tau\varepsilon^{-2}$, integration by parts and Fubini, we derive
\allowdisplaybreaks
 \begin{eqnarray}
 \int\limits_0^t P_sW_L(\tau\varepsilon^{-2})~\txtd \tau &=& \int\limits_0^t\int\limits_0^{\tau\varepsilon^{-2}} e^{(\tau\varepsilon^{-2} -s )L} ~\txtd P_s W(s)~\txtd \tau\nonumber\\
 & =& \int\limits_0^t P_s W(\tau\varepsilon^{-2})~\txtd\tau + L_s\int\limits_0^t \int\limits_0^{\tau\varepsilon^{-2}} e^{(\tau\varepsilon^{-2}-s)L}P_sW(s)~\txtd s ~\txtd \tau\nonumber\nonumber\\
 & =& \int\limits_0^t P_s W(\tau\varepsilon^{-2})~\txtd\tau + L_s\varepsilon^{-2} \int\limits_0^t\int\limits_0^{\tau} e^{(\tau-s)L\varepsilon^{-2}} P_s W(s\varepsilon^{-2})~\txtd s~\txtd \tau\nonumber\\
 & =& \int\limits_0^t P_s W(\tau\varepsilon^{-2})~\txtd\tau + L_s\varepsilon^{-2} \int\limits_0^t \int\limits_s^{t} e^{(\tau-s)L\varepsilon^{-2}} P_sW(s\varepsilon^{-2})~\txtd\tau~\txtd s\nonumber\\
 & =&\int\limits_0^t P_s W(\tau\varepsilon^{-2})~\txtd \tau + \int\limits_0^t P_sW(s\varepsilon^{-2}) e^{(\tau -s)L\varepsilon^{-2}}|_s^t~\txtd s\nonumber\\
 & =& \int\limits_0^t P_s W(s\varepsilon^{-2})~\txtd s +\int\limits_0^t[ e^{(t-s)L\varepsilon^{-2}} P_s W(s\varepsilon^{-2}) -P_sW(s\varepsilon^{-2})]~\txtd s\nonumber\\
 &=& \int\limits_0^t e^{(t-s)L\varepsilon^{-2}} P_sW(s\varepsilon^{-2})~\txtd s,
 \end{eqnarray}
 which finishes the proof.
 \qed\end{proof}
 
 Alternatively, one can obtain~\eqref{pswl} regarding that $W_L$ is the solution of the linear SPDE
 \[
 \txtd u = L u ~\txtd t + \txtd W(t).
 \]
 Applying the stable projection yields
\[
P_sW_L(t) = L_s \int_0^t P_s W_L(s)ds  + P_s W(t).
\]
Thus $\int_0^t P_s W_L(s)ds$ solves 
$\partial_tv=L_s v + P_s W$.
Solving this equation proves the statement, too.

Based on this we establish the order of $P_c\text{Res}$ depending on $H$ as follows.
\begin{lemma}\label{pc:res}
	Let Assumptions~\ref{ass:L}--\ref{incond} hold
	and consider $\tilde\Omega_\varepsilon$ as in Definition \ref{def:Omega}. Then  there exists a constant $C_{\text{\em c,Res}}>0$ such that
	\begin{align}\label{pc:res1}
\tilde\Omega_\varepsilon\subset \Big\{ \omega \in\Omega : 	\sup\limits_{t\in[0,T_0\varepsilon^{-2}] } \|P_c\emph{\Res}(\psi(t))\| \leq C_{\text{\em c,Res}} \varepsilon^{3-13\kappa}\Big\}
\quad&\text{ if } H > \frac{1}{2},
	\end{align}
	respectively 
	\begin{align}\label{pc:res2}
	\tilde\Omega_\varepsilon\subset \Big\{ \omega \in\Omega : 	\sup\limits_{t\in[0,T_0\varepsilon^{-2}] } \|P_c\emph{\Res}(\psi(t))\| \leq C_{\text{\em{c,Res}}} \varepsilon^{2H+1+\frac{2H^2}{1-H}-18\kappa}\Big\}
	\quad&\text{ if } H\leq \frac{1}{2}.	   
	\end{align}
\end{lemma}

\begin{proof}
Recall the ansatz~\eqref{psi}, i.e.
\[ 
\psi(t)=\varepsilon \psi_c(t) + \varepsilon^{2H+1}\psi_s(t), 
\quad\text{where}\quad 
\psi_s(t)=e^{tL}\psi_s(0)+P_sW_L(t)
\]
and $\psi_c(t)=a(\varepsilon^2t)$ where $a$ solves the amplitude equation~\eqref{a:amplitude}. 

Applying $P_c$ to~\eqref{residual} and using the amplitude equation, we obtain
\begin{align}\label{res:c}
P_c\Res(\psi(t))= \int\limits_0^t \varepsilon^2 A_c(P_s\psi(\tau))+ \cF_c(\psi(\tau))-\cF_c(P_c\psi(\tau))~\txtd \tau
\end{align}
By expanding the cubic nonlinearity we derive
\begin{eqnarray}\label{terms:res:c}
P_c\Res(\psi(t)) 
& =&\varepsilon^{2H+3} \int\limits_0^t A_c\psi_s(\tau)~\txtd \tau 
+3\varepsilon^{2H+3} \int\limits_0^t \cF_c(\psi_c(\tau),\psi_c(\tau),\psi_s(\tau))~\txtd\tau \nonumber\\
&& + 3\varepsilon^{4H+3} \int\limits_0^t \cF_c(\psi_c(\tau),\psi_s(\tau),\psi_s(\tau))~\txtd\tau + \varepsilon^{6H+3}\int\limits_0^t \cF_c(\psi_s(\tau))~\txtd\tau \nonumber\\
& =&\varepsilon^{2H+1} \int\limits_0^T A_c\psi_s(\varepsilon^{-2} \tau)~\txtd \tau 
+3\varepsilon^{2H+1} \int\limits_0^T \cF_c(a(\tau),a(\tau),\psi_s(\varepsilon^{-2}\tau))~\txtd\tau \\
&& + 3\varepsilon^{4H+1} \int\limits_0^T \cF_c(a(\tau),\psi_s(\varepsilon^{-2}\tau),\psi_s(\varepsilon^{-2}\tau))~\txtd\tau + \varepsilon^{6H+1}\int\limits_0^T \cF_c(\psi_s(\tau\varepsilon^{-2}))~\txtd\tau \nonumber
\end{eqnarray}
where we used the rescaling to the slow time-scale $T=t\varepsilon^2$.

Using the Assumptions~\ref{ass:A} and~\ref{ass:F} for $A$ and $\mathcal{F}$, together
with the bounds  $\psi_c=\mathcal{O}(\varepsilon^{-2\kappa})$  and 
$\psi_s=\mathcal{O}(\varepsilon^{-\kappa})$ from
Lemmas~\ref{lemma:a} and~\ref{lemma:psi_s}
 we  obtain on $\tilde\Omega_\varepsilon$ 
 for the last two terms in \eqref{terms:res:c}
\[
    \varepsilon^{4H+1} \int\limits_0^T \cF_c(a(\tau), \psi_s(\tau\varepsilon^{-2}),\psi_s(\tau\varepsilon^{-2}))~\txtd\tau =\mathcal{O}(\varepsilon^{4H+1-4\kappa})
\] 
    and
    \[
        \varepsilon^{6H+1} \int\limits_0^{T} \cF_c(\psi_s(\tau\varepsilon^{-2}))~\txtd \tau
        =\mathcal{O}(\varepsilon^{6H+1-3\kappa}).
    \]
For  the first term in~\eqref{terms:res:c} we use the definition of $\psi_s$,
Lemma \ref{lemma:rel:young}, and the fact that $A_c\in\cL(\cN)$. This implies
\begin{eqnarray*}
\varepsilon^{2H+3} \| \int\limits_0^T A_c\psi_s(\varepsilon^{-2} \tau)~\txtd \tau\|  
&\leq &
C \varepsilon^{2H+1} \Big[ 
\|\int\limits_0^T  e^{\tau\varepsilon^{-2} L}\psi_s(0)~\txtd  \tau\|  
+  \| \int\limits_0^T P_s W_L(\varepsilon^{-2}\tau)~\txtd \tau \|\Big] \\
&\leq & C \varepsilon^{2H+1} \Big[ 
\int\limits_0^T  e^{\tau\varepsilon^{-2} \mu} \|\psi_s(0)\| ~\txtd  \tau  
+  \| \int\limits_0^T e^{(T-\tau)\varepsilon^{-2}L}  P_s W(\varepsilon^{-2}\tau)~\txtd \tau \|\Big] \\
&\leq & C \varepsilon^{2H+3} \Big[ \|\psi_s(0)\|  
+  \sup_{\tau \in[0,T]}\| P_s W(\varepsilon^{-2}\tau) \|\Big]\\
& = & \mathcal{O}(\varepsilon^{2H+3-\kappa}) 
\quad\text{on } \tilde\Omega_\varepsilon.
\end{eqnarray*}


We now focus on the final, the second term of~\eqref{terms:res:c}, which is the crucial one in the estimate. 
\begin{align}
\lefteqn{\varepsilon^{2H+1} \int\limits_0^T \cF_c (a(\tau),a(\tau),\psi_s(\tau\varepsilon^{-2}) )~\txtd \tau}\nonumber\\
& = \varepsilon^{2H+1} \int\limits_0^{T}\cF_c (a(\tau),a(\tau), e^{\tau\varepsilon^{-2}L}\psi_s(0) )~\txtd \tau + \varepsilon^{2H+1} \int\limits_0^T \cF_c(a(\tau),a(\tau),  P_sW_L(\tau\varepsilon^{-2}))~\txtd\tau\label{todo}
\\
& := \varepsilon^{2H+1} I_1 +\varepsilon^{2H+1} I_2.\nonumber
\end{align}
We can easily estimate the first integral as follows
\begin{align*}
\|I_1\|\leq \int\limits_0^T \| \cF_c (a(\tau),a(\tau), e^{\tau\varepsilon^{-2}L}\psi_s(0) )\|~\txtd \tau 
&\leq C_a \|a\|^2_{C([0,T];\cN)}\|\psi_s(0)\| \int\limits_0^T C C_a e^{\tau\varepsilon^{-2}\mu}~\txtd\tau 
\\ & \leq\varepsilon^{2-5\kappa} .
\end{align*}
Consequently the first term is of order $\cO(\varepsilon^{2H+3-5\kappa})$ on $\tilde\Omega_\varepsilon$. 

We have several possibilities to estimate $I_2$. 
Directly, we could obtain for $\omega\in\tilde\Omega_\varepsilon$ the upper bound 
\begin{align*}
\|I_2\|&\leq C \int\limits_0^T  \|a(\tau)\|^2 \|P_s W_L(\tau\varepsilon^{-2})\|~\txtd\tau 
 \leq C\varepsilon^{-5\kappa}.
\end{align*}
This would imply that the first term in~\eqref{todo} is of order $\cO(\varepsilon^{2H+1-5\kappa})$, which is not enough for our aims if $H\leq \frac{1}{2}$. 
Therefore, we derive a better estimate under the assumption that $W$ is trace-class with values in $X$. 
Using Lemma~\ref{lemma:rel:young} we can rewrite
 \begin{align*}
\int\limits_0^T \cF_c(a(\tau),a(\tau), P_sW_L(\tau\varepsilon^{-2})) ~\txtd\tau
& =\int\limits_0^T \cF_c(a(\tau ),a(\tau ), \frac{\partial}{\partial\tau} \int\limits_0^{\tau} P_sW_L(s\varepsilon^{-2}))~\txtd s~\txtd \tau\nonumber\\
& =: \int\limits_0^T \mathcal{H}(\tau)~\txtd \mathcal{Z}(\tau),
 \end{align*}
 where we interpret this as a Young-type integral 
 with $\mathcal{Z}(\tau)= \int_0^{\tau} P_sW_L(s\varepsilon^{-2})~\txtd s$ where the integrand $\mathcal{H}$ is bounded by 
 \[
 \|\mathcal{H}\|_{C^{\beta'}([0,T];\cL(X,\mathcal{N}))} \leq 
 C\|a\|^2_{C^{\beta'}([0,T];\mathcal{N})}.
 \]
Therefore we use a Young-type estimate, which follows from  Proposition 1 of \cite{GubinelliLejayTindel} (see also \cite{Young})
 \begin{align}\label{young}
     \Big\|\ \int\limits_0^T \mathcal{H}(\tau)~\txtd \mathcal{Z}(\tau)\Big\| \leq C \|a\|^2_{C^{\beta'}([0,T];\cN)} \cdot \Big\| \int\limits_0^{\tau} P_sW_L(s\varepsilon^{-2})~\txtd s  \Big\|_{C^{\alpha'}([0,T];X)},
 \end{align}
 provided that the sum of the H\"older exponents $\alpha'+\beta'>1$. 
 If $H>\frac{1}{2}$ we can apply this inequality for $\beta'=\beta>\frac{1}{2}$ to further obtain 
 \begin{align*}
 \int\limits_0^T \|\cF_c(a(\tau),a(\tau), P_sW_L(\tau\varepsilon^{-2}))\| ~\txtd\tau \leq C_a \|a\|^2_{C^{\beta}([0,T];\cN)} \cdot \Big\|\tau\mapsto \int\limits_0^\tau P_sW_L(s\varepsilon^{-2})~\txtd s\Big\|_{C^{\beta}([0,T];X)}.
 \end{align*}
  Applying Lemma~\ref{est:h2} leads to 
 \begin{align*}
 \Big\|t\mapsto \int\limits_0^t P_sW_L(\tau\varepsilon^{-2})~\txtd\tau\Big\|_{C^{\beta}([0,T];X)} \leq \varepsilon^{2} \|P_s W(\varepsilon^{-2}\cdot)\|_{C^{\beta}([0,T];X)}.
 \end{align*}
Putting all these estimates together we infer that
\begin{align}\label{i2}
\|I_2\| =  \Big\|\int\limits_0^T \cF_c(a(\tau),a(\tau), P_sW_L(\tau\varepsilon^{-2})~\txtd\tau\Big\|
=\cO(\varepsilon^{2-2H-13\kappa})
\quad\text{ on }\tilde\Omega_\varepsilon,
\end{align}
consequently 
\[
\varepsilon^{2H+1}I_2
=\cO(\varepsilon^{2H+1+2-2H-13\kappa})
=\cO(\varepsilon^{3-13\kappa})
\quad\text{ on }\tilde\Omega_\varepsilon,
\]
which means that the second term in~\eqref{todo} is of order $\cO(\varepsilon^{3-13\kappa})$. This is enough for our aims if the Hurst parameter $H>\frac{1}{2}$. 

For $H\leq \frac{1}{2}$ we apply~\eqref{young} for $\beta'=\beta=H-\kappa$ and $\alpha'=1-H+2\kappa$  and use Lemma~\ref{est:h3}, which relies on a suitable interpolation of the H\"older norms.


Thus we obtain from Lemma~\ref{est:h3} with $\gamma=\beta$ that for an arbitrarily small $\tilde\kappa>0$
\begin{align*}
\int\limits_0^T 
\|\cF_c(a(\tau),a(\tau),P_sW_L(\tau\varepsilon^{-2}))\| ~\txtd\tau 
&\leq C_a \|a\|^2_{C^{\beta}([0,T];\cN)} \cdot \Big\|\tau\mapsto \int\limits_0^\tau P_sW_L(s\varepsilon^{-2})~\txtd s\Big\|_{C^{\alpha'}([0,T];X)}\\
& \leq  C_a \|a\|^2_{C^{\beta}([0,T];\cN)} \varepsilon^{2(1-\alpha')/(1-\beta) - \tilde\kappa} \|P_sW(\varepsilon^{-2}\cdot)\|_{C^\beta([0,T];X)}\\
& \leq C_a\varepsilon^{2(1-\alpha')/(1-\beta)-2H -\tilde\kappa- 13\kappa}
\quad\text{on }\tilde\Omega_\varepsilon.
\end{align*}
Since $\beta=H-\kappa$ and $\alpha'= 1-H+2\kappa$ we obtain that 
\begin{align*}
    2\frac{1-\alpha'}{1-\beta} -2H - \tilde\kappa = \frac{2H^2-2\kappa}{1-H}-\tilde\kappa
    \geq  \frac{2H^2}{1-H}- 5\kappa
\end{align*}
Therefore, we
   get in total that  the second term of~\eqref{todo} is of order $\cO\Big(\varepsilon^{2H+1 +\frac{2H^2}{1-H} -18\kappa}\Big)$, which is consistent with the result for the Brownian motion, see~\cite[Lemma 4.7]{BloemkerHairer1}. Putting all these estimates together proves~\eqref{pc:res1} and~\eqref{pc:res2}.
\qed\end{proof}

The next result justifies the fact that $\psi$ approximates the solution of~\eqref{spde} up to an error term, which is determined by the order of $P_c$Res.
\begin{theorem}\label{approx}
	Let Assumptions~\ref{ass:L}--\ref{incond} hold and consider  $\tilde\Omega_\varepsilon$ as in Definition~\ref{def:Omega}. Furthermore, let $u$ be the maximal mild solution of~\eqref{spde}. Then there exists a positive constant $C_{\text{\em approx}}$ such that for sufficiently small $\varepsilon>0$ we obtain
	\begin{align*}
	&\tilde\Omega_\varepsilon\subset \Big\{\omega\in\Omega :  \sup\limits_{t\in[0,T_0\varepsilon^{-2}]} \|u(t)-\psi(t)\|\leq C_{\text{\em approx}}~ \varepsilon^{3-18\kappa}\Big\}, \quad & \text{if } H> \frac{1}{2}\\
	& \tilde\Omega_\varepsilon\subset \Big\{\omega\in\Omega :  \sup\limits_{t\in[0,T_0\varepsilon^{-2}]} \|u(t)-\psi(t)\|\leq C_{\text{\em approx}}~ \varepsilon^{2H+1+\frac{2H^2}{1-H}-23\kappa}\Big\},\quad &\text{  if } H\leq \frac{1}{2}.
	\end{align*}
\end{theorem}


\begin{proof}

We define $R$ by
\begin{align}\label{r}
u(t) = \psi(t) + \varepsilon^\gamma R(t)=\varepsilon\psi_c(t) +\varepsilon^{2H+1} \psi_s(t) +\varepsilon^{\gamma} R(t),
\end{align}	
where 
\[\gamma=3-13\kappa \text{ if }H> \frac{1}{2}
\quad\text{ or }\quad 
\gamma=2H+1+\frac{2H^2}{1-H}-18\kappa
\text{ if }H\leq 1/2,
\]
which is consistent with the order of the residual.

By definition $R(0)=0$. Our aim is to show that $R$ is of order almost $\cO(1)$, but we need additional $\kappa$'s in the final estimate. 
Taking the difference between~\eqref{mild:sol} and~\eqref{residual}, we obtain
\begin{align*}
R(t) =\varepsilon^{-\gamma} \Res (\psi(t)) +\varepsilon^2\int\limits_0^t e^{(t-\tau)L}  R(\tau)~\txtd\tau + \varepsilon^{-\gamma}\int\limits_0^t e^{(t-\tau)L} [\cF(u(\tau)) -\cF(\psi(\tau))]~\txtd \tau.
\end{align*}
We further set $R_c=P_c R$, $R_s=P_sR$, fix $\delta^*:=5\kappa>0$ such that $2H-3\kappa>\delta^*>4\kappa$ (which is satisfied if $4\kappa<H$) 
and introduce the stopping time
\begin{align}\label{stopping:time:r}
\tau_{R}:=\inf\{ t>0 : \|R(t)\|\geq \varepsilon^{-\delta^*} \}\wedge T_0\varepsilon^{-2}.
\end{align}
Now it remains to show 
that $\tau_{R}\geq T_0\varepsilon^{-2}$.
Regarding that we will bound first $R_s$ and later $R_c$. We have 
\begin{align*}
R_s(t)= \varepsilon^{-\gamma} \text{Res}_s(\psi(t)) + \varepsilon^2 \int\limits_0^t e^{(t-\tau) L_s} A_s R(\tau)~\txtd \tau + \varepsilon^{-\gamma} \int\limits_0^t e^{(t-\tau)L_s}  [\cF_s(u(\tau)) - \cF_s(\psi(\tau))]~\txtd \tau,
\end{align*}
we obtain using the definition of $\tau_R$ and Lemma~\ref{ps:res}
\begin{align*}
\|R_s(t)\|  \leq \ &
 C_{\text{s,Res}} \varepsilon^{-\gamma+3-6\kappa}
\\ & + C \int\limits_0^t(1+(t-\tau)^{-\alpha}) e^{-(t-\tau)\mu} [C_A \varepsilon^{2-\delta^*} +\|\cF(u(\tau))-\cF(\psi(\tau))\|_{X^{-\alpha}} ]~\txtd \tau.
\end{align*}
Using Assumption~\ref{ass:F} and the fact that $u=\psi+\varepsilon^{\gamma}R$, we can estimate $\cF(u)-\cF(\psi)$ as
\begin{align*}
\|\cF(u) - \cF(\psi)\|_{X^{-\alpha}} &\leq C_F \|3\varepsilon^\gamma \psi^2 R + 3\psi \varepsilon^{2\gamma}R^2 +\varepsilon^{3\gamma}R^3 \| \\
& \leq C_F \varepsilon^{\gamma+2} (3 \|\psi \varepsilon^{-1}\|^2 \|R\|  +3 \varepsilon^{\gamma-1} \|\psi \varepsilon^{-1}\| \|R\|^2 + \varepsilon^{2\gamma-2} \|R\|^3 \ ).
\end{align*}
Therefore, for $t\leq \tau_R$ we have on $\tilde\Omega_\varepsilon$ that 
\begin{align*}
\|R_s(t)\| & \leq C_{\text{s,Res}} \varepsilon^{-\gamma+3-6\kappa}+ C \int_0^t (1+(t-\tau)^{-\alpha}) e^{-(t-\tau)\mu} [(C_A +C_F) \varepsilon^{2-\delta^*}]~\txtd \tau ,
\end{align*}
which leads to (using $2-2H -(2H^2)/(1-H)=2/(1-H)$)
\begin{align*}
    \|R_s(t)\| &\leq C \begin{cases}
    \varepsilon^{7\kappa}, & H> 1/2\\
     \varepsilon^{\frac2{1-H} +12\kappa}, & H\leq 1/2
    \end{cases}\\
     & \leq C < \frac12\varepsilon^{-\delta^*} \quad\text{if $\varepsilon$ is small.}
\end{align*}

The final estimate of $R_c$ is more involved. Applying $P_c$ to~\eqref{r} we derive
\begin{align*}
R_c(t)=\varepsilon^{-\gamma} \text{Res}_c(\psi(t)) + \varepsilon^2\int\limits_0^t A_c R_c(\tau)~\txtd \tau + \varepsilon^{-\gamma} \int\limits_0^t [\cF_c(u(\tau)) - \cF_c(\psi(\tau))]~\txtd \tau.
\end{align*}
Recalling that
$$\text{Res}_c(\psi(t))=\int\limits_0^t \varepsilon^2 A_c(\psi(\tau))+\cF_c(\psi(\tau))~\txtd\tau$$
and that $A_c, F_c$ and $\psi$ are continuous in time, we obtain that $\text{Res}_c$ is differentiable. Regarding this and expanding the cubic term, we further get
\begin{align*}
\partial_t R_c& = \varepsilon^2 A_c R_c +\varepsilon^{-\gamma} \partial_t \text{Res}_c(\psi) + \varepsilon^{-\gamma} [\cF_c(\varepsilon\psi_c + \varepsilon^{2H+1} \psi_s + \varepsilon^\gamma R) -\cF_c(\varepsilon \psi_c + \varepsilon^{2H+1} \psi_s ) ].
\end{align*}
Recalling that $a(\varepsilon^2 t)=\psi_c(t)$, rescaling to the slow-time by setting $\widetilde{R}_c(T)=R_c(T\varepsilon^{-2})$, where $\widetilde{R}$ thus evolves on the slow-time scale. Similarly we rescale $R$ (although it is on the fast time-scale) to obtain
\begin{align*}
    \partial_T \widetilde{R}_c
    &= A_c \widetilde{R}_c + \varepsilon^{-\gamma}\partial_T (\text{Res}_c(\psi(\varepsilon^{-2}T))) \\
    & \qquad + \varepsilon^{-\gamma-2} [\cF_c(\varepsilon\psi_c +\varepsilon^{2H+1} \psi_s + \varepsilon^\gamma \widetilde{R}) -\cF_c(\varepsilon\psi_c+\varepsilon^{2H+1}\psi_s)]\\
    & =  A_c \widetilde{R}_c + \varepsilon^{-\gamma}\partial_T (\text{Res}_c(\psi(\varepsilon^{-2}T))) )+ 3\varepsilon^{-\gamma+1}\cF_c(a,a,\varepsilon^{\gamma-1}\widetilde{R})\\
    &  \qquad+ \varepsilon^{-\gamma-2} [\cF_c(\varepsilon\psi_c +\varepsilon^{2H+1} \psi_s + \varepsilon^\gamma \widetilde{R}) 
    -\cF_c(\varepsilon\psi_c+\varepsilon^{2H+1}\psi_s) 
     - 3 \varepsilon^{-\gamma+1}\cF_c(a, a, \varepsilon^{\gamma-1} \widetilde{R})]\\
    & = [A_c + 3\cF_c(a,a, \widetilde{R})] \widetilde{R}_c + \varepsilon^{-\gamma}\partial_T (\text{Res}_c(\psi(\varepsilon^{-2}T))\\
    &  \qquad+  \varepsilon^{-\gamma+1} [\cF_c(\psi_c +\varepsilon^{2H} \psi_s + \varepsilon^{\gamma-1} \widetilde{R}) -\cF_c(\psi_c+\varepsilon^{2H}\psi_s) - 3 \cF_c(a, a, \varepsilon^{\gamma-1} \widetilde{R})].
\end{align*}
We introduce the $\cL(\cN)$-valued process $B_a$ by $B_a(T)v:=3 F_c(a(T),a(T),v)$ for $T>0$ and $v\in\cN$. 
Moreover, we abbreviate the remaining nonlinear terms as 
\begin{align}\label{KT}
\varepsilon^{\gamma-1} \cK(T):=[\cF_c(\psi_c +\varepsilon^{2H} \psi_s + \varepsilon^{\gamma-1} \widetilde{R}) -\cF_c(\psi_c+\varepsilon^{2H}\psi_s)- 3 \cF_c(\psi_c, \psi_c, \varepsilon^{\gamma-1} \widetilde{R})].   
\end{align}
Consequently, the previous equation for $\widetilde{R}_c$ rewrites as
\begin{align}\label{nonautonomous:eq}
\partial_T \widetilde{R}_c(T) =[A_c + B_a(T)] \widetilde{R}_c(T) + \varepsilon^{-\gamma}\partial_T (\text{Res}_c(\psi(\varepsilon^{-2}T))) + \cK(T).
\end{align}
Note that this is a non-autonomous evolution equation in $\cN$, which we have to analyze in order to bound $\widetilde{R}_c$. Major technical problems arise due to the fact that $B_a$ is not of order $1$. 
Thus we cannot simply apply Gronwall inequality, where the bound on $B_a$ would appear in the exponential.

Setting $d:=\psi_c + \varepsilon^{2H}\psi_s$ for short hand notation and expanding the nonlinear term results in
\begin{align*}
    \lefteqn{\cF_c(\psi_c+ \varepsilon^{2H} \psi_s + \varepsilon^{\gamma-1} R) - \cF_c(\psi_c + \varepsilon^{2H}\psi_s) - 3 \cF_c(\psi_c,\psi_c, \varepsilon^{\gamma-1} R)} \\
    &= \cF_c(d+ \varepsilon^{\gamma-1} R, d+ \varepsilon^{\gamma-1} R, d+ \varepsilon^{\gamma-1} R  ) - \cF_c(d,d,d) - 3 \cF_c(\psi_c,\psi_c, \varepsilon^{\gamma-1}R)  \\
    & = 3 \varepsilon^{\gamma-1} \cF_c(d,d, R)  +3\varepsilon^{2\gamma-2} \cF_c(d,R,R) + \varepsilon^{3\gamma-3} \cF_c(R,R,R) \\
    & = 3  \varepsilon^{\gamma-1} \cF_c(\psi_c + \varepsilon^{2H} \psi_s, \psi_c + \varepsilon^{2H}\psi_s,  R) + 3\varepsilon^{2\gamma-2} \cF_c(\psi_c +\varepsilon^{2H}\psi_s, R, R) \\ 
    & \qquad\qquad\qquad\qquad + \varepsilon^{3\gamma-3} \cF_c(R,R,R)  
    - 3\cF_c(\psi_c,\psi_c, \varepsilon^{\gamma-1} R)\\
    & = 6 \varepsilon^{2H+\gamma-1}\cF_c(\psi_c,\psi_s, R) + 3 \varepsilon^{4H+\gamma-1}\cF_c( \psi_s, \psi_s,   R) +3\varepsilon^{2\gamma-2}\cF_c(\psi_c,R, R) \\
    & \qquad\qquad\qquad\qquad  +3 \varepsilon^{2H+2\gamma-2} \cF_c(\psi_s, R,R) + 3\varepsilon^{3\gamma-3}\cF_c(R,R,R) .
\end{align*}
Now we compute the order of $\cK$ regarding~\eqref{KT} (just divide the previous by $\varepsilon^{\gamma-1}$ to get to $\cK(T)$).
We observe that most of the previous terms are of higher order. 
Note that the third term gives a contribution of order  $\cO(\varepsilon^{2H+\frac{2H^2}{1-H}-19\kappa-2\delta^*})$ for $H\leq \frac{1}{2}$, which shows once more that the estimate on $P_c\text{Res}$ for $H\leq \frac{1}{2}$ derived in Lemma~\ref{pc:res} is crucial. 
For all the terms we use that $\cF$ is trilinear, the uniform bound on $a$ and $\psi_s$ from Lemmas~\ref{lemma:a} and \ref{lemma:psi_s}, and the definition of the stopping time $\tau_R$ specified in~\eqref{stopping:time:r}, in order to bound $R$.
Putting all these estimates together we infer that
\begin{align*}
\|\cK(T)\|= \cO(\varepsilon^{2H-3\kappa-\delta^*}), \quad\text{ provided that } T\in[0,T_0]\cap[0,\tau_R\varepsilon^2].
\end{align*}
Furthermore, we denote by $U(\cdot,\cdot)$ the evolution family generated by the non-autonomous linear operator of~\eqref{nonautonomous:eq} $A_c+B_a(T)$. Using that $A_c\in\cL(\cN)$ and that $\langle v, B_a(T)v\rangle< 0$, according to assumption~\eqref{f1}, we have  $\langle v, (A_c+B_a(T))v\rangle\leq \|A_c\|_{\cL(\cN)}\|v\|^2$ for all $v\in\cN$. Therefore using Gronwall's lemma we conclude that
\begin{align}\label{est:U}
\|U(T,S)\|_{\cL(\cN)} \leq e^{\|A_c\|_{\cL(\cN)}(T-S)} \quad \text{ for all } 0\leq S\leq T.
\end{align}
Solving~\eqref{nonautonomous:eq} using the variation of constants formula further leads to 
\begin{align}\label{var:const}
   \widetilde{R}_c(T) =\varepsilon^{-\gamma} \int\limits_0^T U(T,S)\partial_S(\text{Res}_c(\psi(\varepsilon^{-2}S) ))~\txtd S + \int\limits_0^T U(T,S)\cK(S)~\txtd S.
\end{align}
Using the bound on $\cK$ together with \eqref{est:U} the second term of the previous equality can be estimated as
\begin{align*}
    \Big\| \int\limits_0^T U(T,S)\cK(S)~\txtd S \Big\|\leq C_{T_0} \varepsilon^{2H-3\kappa-\delta^*} \leq C.
\end{align*}
In order to estimate the first term in~\eqref{var:const} one applies the
integration by parts formula 
\begin{align*}
    \varepsilon^{-\gamma} \int\limits_0^T U(T,S)\partial_S(\text{Res}_c(\psi(\varepsilon^{-2}S)))~\txtd S 
    &=\varepsilon^{-\gamma} \text{Res}_c(\psi(\varepsilon^{-2}T)) 
    - \varepsilon^{-\gamma} U(T,0)\text{Res}_c(\psi(0)) \\
    & \qquad + \varepsilon^{-\gamma} \int\limits_0^T U(T,S)(A_c + B_a(S))\text{Res}_c(\psi(\varepsilon^{-2}S)))~\txtd S \\
    & \leq C\varepsilon^{-4\kappa},
\end{align*}
where we combined the bounds on $A_c+B_a(T)$ together with~\eqref{est:U} and the bound on the residual by $\varepsilon^\gamma$. 
In conclusion,
we infer due to Lemma~\ref{pc:res} that 
\[
\|R_c(t)\|\leq C_{T_0}\varepsilon^{-4\kappa} <\frac12 \varepsilon^{-\delta^*}
    \quad\text{ for }t\leq\min\{ \tau_R, \varepsilon^{-2}T_0 \}.
\]    
    By the bound on $R_s$, the continuity of $R$ and the definition of $\tau_R$, this implies that $\tau_R\geq \varepsilon^{-2}T_0$. We thus finished the proof.
\qed\end{proof}


\appendix
\section{H\"older type-estimates of deterministic convolutions}\label{appendix:a}
We provide several possibilities to estimate the convolution \begin{align*}
 \Bigg\|t\mapsto \int\limits_0^t e^{(t-s)L\varepsilon^{-2}} f(s)~\txtd s\Bigg\|_{C^{\alpha}([0,T];X)},
\end{align*} 
where in our applications $f(s):=P_s W(s\varepsilon^{-2})$ and $f(0)=0$. 
There are several optimal regularity results for such convolutions ~\cite{Sinestrari}, but we provide the computations because we need the precise dependence on $\varepsilon$.

\begin{lemma}\label{est:h1}
For $\alpha\in[0,1)$ and $f\in C([0,T];X)$
	we have 
	\begin{align}\label{h1}
	\Bigg\|t\mapsto \int\limits_0^t e^{(t-s)L\varepsilon^{-2}} f(s)~\txtd s\Bigg\|_{C^{\alpha}([0,T];X)} \leq C \varepsilon^{2-2\alpha}\|f\|_{C([0,T];X)}.
	\end{align}
\end{lemma}
\begin{proof}
The case $\alpha=0$ is trivial, so we only focus on the case 
 $\alpha\in(0,1)$.
 
	Let $t,\tau\in[0,T]$ such that $t+\tau\leq T.$ We compute
	\begin{align*}
	&\int\limits_0^{t+\tau} e^{(t+\tau-s)L\varepsilon^{-2}} f(s)~\txtd s - \int\limits_0 ^{t} e^{(t-s)L\varepsilon^{-2}}f(s)~\txtd s\\
	& = \int\limits_t^{t+\tau} e^{(t+\tau-s)L\varepsilon^{-2}} f(s)~\txtd s + (e^{\tau L\varepsilon^{-2}}-\mbox{Id}) \int\limits_0^t e^{(t-s)L\varepsilon^{-2}} f(s)~\txtd s\\
	& = I_1 + I_2.
	\end{align*}
	We estimate
	\begin{align*}
	\|I_1\| &\leq M \|f\|_{C([0,T];X)} \int\limits_t^{t+\tau} e^{-\mu(t+\tau-s)\varepsilon^{-2}}~\txtd s\\
	& \leq M \varepsilon^2 \|f\|_{C([0,T];X)}  \int\limits_0^{\tau\varepsilon^{-2}}e^{-\mu y }~\txtd y \\
	& \leq  M \varepsilon^2 \min\{\mu^{-1}, \tau\varepsilon^{-2}\} \|f\|_{C([0,T];X)} ,
	\end{align*}
	where we used the substitution $y=(\tau -s)\varepsilon^{-2}$. For the second term we have
	\begin{align*}
	\|I_2\| &\leq \| e^{\tau L \varepsilon^{-2}} -\mbox{Id}\|_{\cL(X^{\alpha},X)} \Big\| \int\limits_0^t e^{(t-s)L\varepsilon^{-2}}f(s)~\txtd s \Big\|_{X^{\alpha}}\\
	& \leq M \|f\|_{C([0,T];X)} (\tau\varepsilon^{-2})^{\alpha} \int\limits_0^t (t-s)^{-\alpha}\varepsilon^{2\alpha} e^{-\mu (t-s)\varepsilon^{-2}}~\txtd s\\
	& \leq M  \|f\|_{C([0,T];X)} \tau^{\alpha}\varepsilon^{2-2\alpha}\int\limits_0^{t\varepsilon^{-2}} y ^{-\alpha} e^{-\mu y}~\txtd y.
	\end{align*}
	Combining these two estimates proves the statement. \qed
	\end{proof}
	
We can improve the order in $\varepsilon$ of the bound, if we consider H\"older-continuous $f$.
	
\begin{lemma}\label{est:h2}
	For $\alpha\in(0,1)$  and $f\in  C^{\alpha}([0,T];X)$ with $f(0)=0$ we have
	\begin{align*}
	\Bigg\|t\mapsto \int\limits_0^t e^{(t-s)L\varepsilon^{-2}} f(s)~\txtd s\Bigg\|_{C^{\alpha}([0,T];X)}\leq C \varepsilon^2 \|f\|_{C^{\alpha}([0,T];X)}.
	\end{align*}
\end{lemma}
\begin{proof}
	Let $t,\tau\in[0,T]$ such that $t+\tau\leq T$. Since $f(0)=0$ we obtain 
	\begin{align*}
	&\int\limits_0^{t+\tau} e^{(t+\tau -s)L\varepsilon^{-2}} f(s)~\txtd s - \int\limits_0^t e^{(t-s)L\varepsilon^{-2}}f(s)~\txtd s\\
	& = \int\limits_0^t e^{(t-s)L\varepsilon^{-2}} [f(s+\tau) -f(s) ]~\txtd s + \int\limits_0^{\tau} e^{(t+\tau -s) L\varepsilon^{-2}} [f(s)-f(0)]~\txtd s := I_1 + I_2.
	\end{align*}
	This further results in 
	\begin{align*}
	\|I_1\| &\leq M \int\limits_0^t e^{-\mu(t-s)\varepsilon{^{-2}}} \tau ^{\alpha} \|f\|_{\alpha}~\txtd s \leq M \tau ^{\alpha} \|f\|_{\alpha} \int\limits_0^t e^{-\mu(t-s)\varepsilon^{-2}}~\txtd s\\
	& \leq M \varepsilon^{2}\tau^{\alpha} \|f\|_{\alpha} \int\limits_0^{t\varepsilon^{-2}} e^{-\mu y}~\txtd y 
	\end{align*}
	and 
	\begin{align*}
	\|I_2\|\leq & M\int\limits_0^{\tau} e^{-\mu\varepsilon^{-2}(t+\tau -s )} s^{\alpha} \|f\|_{\alpha}~\txtd s \leq M \tau^{\alpha}\|f\|_{\alpha} \int\limits_0^{\tau} e^{-\mu\varepsilon{^{-2}}(t+\tau -s )}~\txtd s\\
	& \leq M\tau^{\alpha} \varepsilon^2 \|f\|_{\alpha} e^{-\mu \varepsilon^{-2}t} \int\limits_0^{\tau\varepsilon^{-2}} e^{-\mu y}~\txtd y ,
	\end{align*}
	which immediately proves the statement.
	\qed
	\end{proof}

Now we combine the two previous Lemmas by interpolation to bound 
	\[
	(*):=\Big\|t\mapsto \int\limits_0^t e^{(t-s)L\varepsilon^{-2}} f(s)~\txtd s\Big\|_{C^{\alpha}([0,T];X)} \quad \text{
by}\quad  	\|f\|_{C^{\alpha-\gamma}([0,T];X)}.
	\]
Choose for $0<\gamma<\alpha<\beta<1$ 
\[
p=\frac{\beta-\gamma}{\beta-\alpha}
\quad\text{and}\quad
q=\frac{\beta-\gamma}{\alpha-\gamma}
\quad\text{with}\quad
\frac1p+\frac1q=1
\quad\text{and}\quad
\frac{\gamma}{p}+\frac{\beta}q= \alpha.
\]
Thus the  the $\alpha$-H\"older norm is bounded now by (first interpolation, then Lemma \ref{est:h1} and \ref{est:h2})
\begin{align*}
\|(*)\|_{C^\alpha([0,T],X)} 
&\leq C \|(*)\|_{C^\beta([0,T],X)}^{1/q} \|(*)\|_{C^\gamma([0,T],X)}^{1/p} \\
& \leq C
\left(	\varepsilon^{2-2\beta}\|f\|_{C^0} \right)^{1/q}
\left(	\varepsilon^2\|f\|_{C^\gamma}\right)^{1/p}  \\
&\leq C
\varepsilon^{2(\frac1p+(1-\beta)\frac1q)}\|f\|_{C^\gamma} = C
\varepsilon^{2(1-\frac{\beta}q)}\|f\|_{C^\gamma}.
\end{align*}
Now we can maximize the exponent by setting $\beta=1$ to obtain
\[
2\Big(1-\frac{\beta}q\Big)=
2\Big(1-\frac{\alpha-\gamma}{1-\gamma}\Big)=2\frac{1-\alpha}{1-\gamma}.
\]
Obviously, we need $\beta<1$, so we can only get arbitrarily close to the maximum.
We proved

\begin{lemma}\label{est:h3}
	For $\alpha,\gamma\in[0,1)$ with $\alpha\geq\gamma$  and $f\in  C^{\gamma}([0,T];X)$ with $f(0)=0$ we have
	for any $0\leq\zeta<2(1-\alpha)/(1-\gamma)$ 
	a constant $C_\zeta>0$ such that
	\begin{align*}
	\Big\|t\mapsto \int\limits_0^t e^{(t-s)L\varepsilon^{-2}} f(s)~\txtd s\Big\|_{C^{\alpha}([0,T];X)}\leq C_\zeta \varepsilon^{\zeta} \|f\|_{C^{\gamma}([0,T];X)}.
	\end{align*}
\end{lemma}

\begin{remark} We use this result in Lemma~\ref{pc:res} to prove an optimal bound for the residual on the slow modes for small values of the Hurst parameter. For the Brownian motion, a direct interpolation of suitable H\"older norms suffices for the proof of Lemma~\ref{pc:res}. 
\end{remark}
\section{On the stochastic convolution with fractional noise}\label{stoch:conv}
%

We let $K>0$ and $T>0$ and provide an estimate for the probability $\mathbb{P}(\sup\limits_{t\in[0,T]} \|P_sW_L(t)\|>K )$, which is used in Section~\ref{s:mainresults} (see Remark~\ref{rem:SC}) for $T=T_0\varepsilon^{-2}$ and $K=\varepsilon^{-\kappa}$, for an arbitrarily small $\kappa> 0$ and $T_0>0$.  
\begin{remark}
For $H=\frac{1}{2}$ such estimates are well-known and immediately follow using Chebyshev's inequality and the factorization lemma~\cite{DaPratoZ}. In this case for all $p\geq 1$ one obtains
\begin{align*}
\mathbb{P}(\sup\limits_{t\in[0,T]} \|P_sW_L(t)\|)> K )\leq C_p K^{-p} T.
\end{align*}
Choosing $T=T_0\varepsilon^{-2}$ and $K=\varepsilon^{-k}$, this further leads to
$$\mathbb{P}(\sup\limits_{t\in[0,T_0\varepsilon^{-2}]} \|P_sW_L(t)\|>\varepsilon^{-\kappa} ) \leq C_p (\varepsilon^{\kappa p})T_0\varepsilon^{-2}.$$
Consequently, since $p$ can be taken arbitrarily large, in order to make this probability small for $H=\frac{1}{2}$, we need a small $\kappa>0$, as pointed out in Remark~\ref{k:rem}.
\end{remark}
In order to prove similar statements for the fractional Brownian motion we consider separately the cases $H\in(\frac{1}{2},1)$ and $H\in(0,\frac{1}{2})$. 
We recall that $(W^H(t))_{t\geq 0}$ is a trace-class fractional Brownian motion, i.e.
\begin{align*}
    W^H(t)=\sum\limits_{k=1}^{\infty}\sqrt{q_k} \beta^H_k(t)e_k,
\end{align*}
where $\text{tr}~Q=\sum\limits_{k=1}^{\infty}q_k<\infty$, $\{\beta^H_k(\cdot)\}_{k\geq 1}$ are independent standard fractional Brownian motions and $\{e_k\}_{k\geq 1}$ is an orthonormal basis of $X$. 

Further we impose the stronger assumption that  $\{e_k\}_{k\geq 1}$ are also eigenfunctions of $L$ and $\{\lambda_k\}_{k\geq 1}$ the corresponding eigenvalues.
We believe that one can prove the result without that assumption, but for simplicity of presentation, we assume this here.

Analogously to the Brownian case, we rely on the celebrated factorization method and therefore
fix $\alpha\in(0,H)$ and introduce the stochastic process
\begin{align*} 
    Y(t)&:=P_s \int\limits_0^t (t-r)^{-\alpha} e^{(t-r)L}~\txtd W^H(r)\\
    & =\sum\limits_{k=1}^{\infty} \sqrt{q_k} \int\limits_0^t (t-r)^{-\alpha} e^{(t-r)\lambda_k}~\txtd \beta^H_k(r) e_k.
\end{align*}
Now the stochastic convolution can be written as
\begin{align*}
P_s W_L(t) =\frac{\sin \pi\alpha}{\pi} \int\limits_0^t (t-r)^{\alpha-1} P_s e^{(t-r)L}Y(r)~\txtd r,
\end{align*}
see~\cite{DuncanMaslowski} for $H>1/2$ respectively~\cite{TTV} for $H<1/2$.

Due to Chebyshev's and H\"older inequality and the exponential stability of the semigroup $(e^{tL})_{t\geq 0}$ on $\cS$, we have for $p\geq 1$ that
\begin{align*}
\P(\sup\limits_{t\in[0,T_0]} \|P_s W_L(t)\|>K )& \leq  C_pK^{-p} \E( \sup\limits_{t\in[0,T_0]} \|P_sW_L(t)\|^p )\\
& = C_pK^{-p} \E\sup\limits_{t\in[0,T_0]} \Bigg\| \frac{\sin \pi\alpha}{\pi} \int\limits_0^t (t-r)^{\alpha-1} P_s e^{(t-r)L}Y(r)~\txtd r   \Bigg\|^p\\
& \leq C_pK^{-p} \E \sup\limits_{t\in[0,T_0]}\Bigg[ \Bigg( \int\limits_0^t r^{(\alpha-1)\frac{p}{p-1}}e^{-\mu r\frac{p}{p-1}}   ~\txtd r\Bigg)^{p-1} \int\limits_0^{t}\|Y(r)\|^p~\txtd r\Bigg]\\
& \leq C_p K^{-p} \int\limits_0^{T_0} \E\|Y(r)\|^p~\txtd r\\
& \leq C_p K^{-p} T_0 \sup\limits_{t\in[0,T_0]}\E\|Y(t)\|^p.
\end{align*}
Due to Gaussianity, it is enough to compute only the second moment of $Y$. To this aim we treat the cases $H>\frac{1}{2}$ and $H<\frac{1}{2}$ separately. For the Brownian motion, it is well-known that $\sup\limits_{t\in[0,T_0]}E\|Y(t)\|^2 \leq \widetilde{C}$, where the constant $\widetilde{C}$ is universal and does not depend on time.

\begin{lemma}\label{conv1} Let $H\in(\frac{1}{2},1)$. Then 
\begin{align*}
\E\|Y(t)\|^2 \leq \widetilde{C}(H), ~\text{ for all } t>0.
\end{align*}
    \end{lemma}

\begin{proof}
We compute
\begin{align*}
    \E \|Y(t)\|^2& =\sum\limits_{k=1}^{\infty} q_k \E\Bigg[\int\limits_0^t (t-r)^{-\alpha} e^{(t-r)\lambda_k}~\txtd \beta^H_k(r) \Bigg]^2 \\
    & \leq C(H) \sum\limits_{k=1}^{\infty} q_k \int\limits_0^t \int\limits_0^t (t-u)^{-\alpha} (t-v)^{-\alpha} e^{(t-u)\lambda_k} e^{(t-v)\lambda_k} |u-v|^{2H-2}~\txtd u ~\txtd v,
\end{align*}
where we used~\cite[Section 3]{DuncanMaslowski} to estimate the stochastic convolution. We used $C(H):=H(2H-1)$.
Using the exponential stability on $\cS$ of $(e^{tL})_{t\geq 0}$, 
we evaluate for an arbitrary $\lambda>0$ the expression
\begin{align*}
    &\int\limits_0^t \int\limits_0^t u^{-\alpha} v^{-\alpha} e^{-\lambda u}e^{-\lambda v} |u-v|^{2H-2}~\txtd u~\txtd v\\
    & \leq 2 \int\limits_0^t \int\limits_0^v u^{-\alpha} v^{-\alpha} e^{-\lambda u}e^{-\lambda v }(v-u)^{2H-2}~\txtd u~\txtd v\\
    & = 2 \int\limits_0^t v^{-\alpha} e^{-\lambda v} \int\limits_0^v u^{-\alpha} \underbrace{e^{-\lambda u}}_{\leq 1} (v-u)^{2H-2}~\txtd u~\txtd v.
\end{align*}
We perform the substitution $w=\frac{u}{v}$ and further obtain
\begin{align*}
&\int\limits_0^t v^{-\alpha} e^{-\lambda v} \int\limits_0^1 (wv)^{-\alpha} (v-w v )^{2H-2} v~\txtd w~\txtd v\\
& = \int\limits_0^t v^{-2\alpha} e^{-\lambda v } \int\limits_0^1 w^{-\alpha} v^{2H-1}(1-w)^{2H-2}~\txtd w~\txtd v\\
& = \int\limits_0^t v^{-2\alpha} e^{-\lambda v } v^{2H-1} \underbrace{\int\limits_0^1 w^{-\alpha} (1-w)^{2H-2}~\txtd w}_{=B(-\alpha+1,2H-1)}\txtd v\\
& \leq B(-\alpha+1,2H-1) \int\limits_0^{\infty} v^{2(H-\alpha)-1} e^{-\lambda v}~\txtd v.
\end{align*}
Here $B(\cdot,\cdot)$ is Euler's Beta function and the last integral is well-defined due to the fact that $2H-2\alpha-1>-1$ since $\alpha<H$. Consequently
\begin{align*}
\E \|Y(t)\|^2 \leq C(H) B(-\alpha+1,2H-1) \text{tr}~Q.
\end{align*}
  \qed  
\end{proof}

For $H<\frac{1}{2}$ the estimates are more involved, compare~\cite{TTV,DuncanMaslowski2}. We first follow~\cite{TTV} to show a bound on $Y$ for $t$ away from $0$.

\begin{lemma}\label{conv2}
Let $H\in(0,\frac{1}{2})$. Then 
\begin{align*}
\E\|Y(t)\|^2 \leq \widetilde{C}(H),~\text{ for all } t\geq 1.
\end{align*}
\end{lemma}
 
\begin{proof}
We estimate the second moment as follows
\begin{align*}
\E \|Y(t)\|^2 &= \sum\limits_{k=1}^{\infty} q_k \E\Bigg[ \int\limits_0^t (t-s)^{-\alpha}e^{(t-s)\lambda_k} ~\txtd\beta^H_k(s)\Bigg]^2\\
& \leq 2 \sum\limits_{k=1}^{\infty} q_k\int\limits_0^t (t-s)^{-2\alpha} e^{2(t-s)\lambda_k} K^2(t,s)~\txtd s\\
&+ 2 \sum\limits_{k=1}^{\infty}q_k \int\limits_0^t \int\limits_s^t|[ (t-r)^{-\alpha} e^{(t-r)\lambda_k} -(t-s)^{-\alpha}e^{(t-s)\lambda_k} ] \frac{\partial K}{\partial r}(r,s)|^2~\txtd r~\txtd s\\
& = :I_1 + I_2.
\end{align*}
Here, the fractional kernel $K$ is given by (see~\cite{D})
\begin{align*}
K(t,s)= c_H (t-s)^{H-\frac{1}{2}} + s^{H-\frac{1}{2}} F\Big(\frac{t}{s}\Big),
\end{align*}
where $c_H>0$ is a constant and
\begin{align*}
    F(z)= c_H \Big(\frac{1}{2}-H\Big)\int_0^{z-1} r^{H-\frac{3}{2}} (1- (1+r)^{H-\frac{1}{2}})~\txtd r.
\end{align*}
The estimate of $I_1$ relies on the following property of the kernel (\cite[Theorem 3.2]{D}) 
\begin{align}\label{k:op}
    K(t,s)\leq c(H) (t-s)^{H-\frac{1}{2}} s^{H-\frac{1}{2}}.
\end{align}

In order to estimate $I_1$ we  observe that for $\lambda>0$
\begin{align*}
C(H) \int\limits_0^t e^{-2\lambda(t-s)}(t-s)^{2(H-\alpha)-1}s^{2H-1}~\txtd s 
&=C(H)\lambda^{-2(H-\alpha)} \int\limits_0^{2\lambda t} e^{-v}v^{2H-1} (t-\frac{v}{2\lambda})^{2(H-\alpha)-1}~\txtd v,
\end{align*}
where we use the substitution $v=\frac{t-s}{2\lambda}$. This further entails
\begin{align*}
&\int\limits_0^{2\lambda t} e^{-v}v^{2H-1} (t-\frac{v}{2\lambda})^{2(H-\alpha)-1}~\txtd v\\
&= \int\limits_0^{\lambda t } e^{-v}v^{2H-1} (t-\frac{v}{2\lambda})^{2(H-\alpha)-1}~\txtd v + \int\limits_{\lambda t }^{2\lambda t }  e^{-v}v^{2H-1} (t-\frac{v}{2\lambda})^{2(H-\alpha)-1}~\txtd v\\
& \leq (t/2)^{2(H-\alpha)-1} \int\limits_0^{\infty} e^{-v} v^{2H-1}~\txtd v + (\lambda t)^{2(H-\alpha)-1} \int\limits_{\lambda t }^{2\lambda t }e^{-v}(t-\frac{v}{2\lambda})^{2(H-\alpha)-1}~\txtd v\\
& \leq C(H) t^{2(H-\alpha)-1} + (\lambda t )^{2(H-\alpha)-1} \int\limits_0^{\lambda t} e^{-(2\lambda t -v')} (v'/2\lambda)^{2(H-\alpha)-1}~\txtd v'\\
& \leq C(H) t^{2(H-\alpha)-1} +C(H)(\lambda t )^{2(H-\alpha)-1} e^{-\lambda t } (\lambda t )^{2(H-\alpha)}.
\end{align*}


Now the second term will become small for $t\to\infty$ but we still have a $t^{2(H-\alpha)-1}$ due to the first estimate. However, for large times the term $I_2$ is dominant. To this aim, one uses that $\frac{\partial K}{\partial r}(r,s)\leq 0$ for every $ r,s\in(0,t]$ together with $|\frac{\partial K}{\partial r}(r,s)|\leq C(H)(r-s)^{H-\frac{3}{2}}$. Regarding this, one can infer similarly to~\cite[Theorem 4]{TTV} that
 \begin{align*}
 & \int\limits_0^t \int\limits_s^t|[ (t-r)^{-\alpha} e^{-\lambda (t-r)} -(t-s)^{-\alpha}e^{-\lambda(t-s)} ] \frac{\partial K}{\partial r}(r,s)|^2~\txtd r~\txtd s \\
 &\leq C(H)\int_0^t \int_0^u \int_0^ u (u-v_1)^{H-3/2}(u-v_2)^{H-3/2} \Big( (v_1v_2)^{-\alpha} e^{-\lambda (v_1+v_2)}  -(v_1 u)^{-\alpha}e^{-\lambda (v_1+u)}\\
 &\hspace*{76 mm}-(u v_2)^{-\alpha}e^{-\lambda (u+v_2) } + u^{-2\alpha} e^{-2\lambda u} \Big ) ~\txtd v_2~\txtd v_1~\txtd u \\
 & \leq C(H)  \int_0^t \Bigg( \int_0^u (u-v)^{H-3/2} (u^{-\alpha} e^{-\lambda u} - v^{-\alpha} e^{-\lambda v}) ~\txtd v\Bigg)^2~\txtd u\\
 &\leq \widetilde{C}(H) \lambda^{-2(H-\alpha)}\int_0^{\lambda t} \Bigg( \int_0^u (u-v)^{H-3/2} (u^{-\alpha} e^{- u} - v^{-\alpha} e^{-v}) ~\txtd v\Bigg)^2~\txtd u,
 \end{align*}
 since \begin{align*}
     \sup_{t\geq 0 }\int_0^{\lambda t} \Bigg( \int_0^u (u-v)^{H-3/2} (u^{-\alpha} e^{- u} - v^{-\alpha} e^{-v}) ~\txtd v\Bigg)^2~\txtd u <\widetilde{C},
 \end{align*}
 according to~\cite[Theorem 3 and Appendix]{TTV}.
\qed\end{proof}

\begin{remark}
The result in Lemma~\ref{conv2} is optimal for $t\geq 1$. Due to the estimate of the fractional kernel in~\eqref{k:op} a singularity occurs in the upper bound of $I_1$ as $t\to 0$. 
For $t\in(0,1]$ the assertion can be proved as in~\cite[Theorem 11.11]{DuncanMaslowski2}. This follows by a computation which uses the structure of the fractional kernel in order to derive an estimate on the stochastic convolution~\cite[Theorem 11.8]{DuncanMaslowski2}. More precisely, for an analytic semigroup $S(s)=e^{sL}$, a trace-class fractional Brownian motion $W^H(\cdot)$ and $\beta\in(\frac{1}{2}-H,\frac{1}{2})$, the following estimate holds true for a positive constant $c=c(\beta)$ (\cite[Theorem 11.8]{DuncanMaslowski2})
\begin{align*}
    \E\Big\| \int_0^T S(t-s)~\txtd W^H(s) \Big\|^2\leq c \int_0^T \Bigg(\frac{\|S(s)\|^2}{(T-s)^{1-2H}} + \frac{\|S(s)\|^2}{s^{2\beta}}\Bigg) ~\txtd s.
\end{align*}
The additional pole $(t-s)^{-\alpha}$ in $Y$ does not change this result very much, only the order of the poles on the right hand side might be modified by an $\alpha$, which can be chosen arbitrarily small. 
But we refrain from repeating the whole argument and all estimates of \cite{DuncanMaslowski2}.

Nevertheless this estimate gives bounds that are unbounded in $T$ if $T$ is large. Therefore for large times, a slightly different argument is required~\cite[Theorem 11.13]{DuncanMaslowski2}, or the computation in Lemma \ref{conv2}.
\end{remark}

\begin{remark}
Note that for the computations in Lemmas~\ref{conv1} and \ref{conv2}, the assumption $\text{tr}~Q<\infty$ is not necessary and can be replaced by $\sum_{k=1}^{\infty} q_k |\lambda_k|^{-2(H-\alpha)}<\infty$ for $\alpha\in[0,H)$. However, we use the trace-class condition on the noise for the H\"older estimates required in Section~\ref{s:mainresults}. 

\end{remark}




\end{document}